\documentclass{amsart}[10pt]

\usepackage{amsrefs}
\usepackage[all]{xy}
\usepackage{syntonly}
\usepackage{hyperref}
\usepackage{amsfonts}
\usepackage{amssymb}
\usepackage{amsmath}
\usepackage{amsthm}
\usepackage[inline]{enumitem}
\usepackage{tikz-cd}
\usepackage{graphics}
\usepackage{graphicx}
\usepackage{color}
\usepackage{comment}
\usepackage{caption,lipsum}
\usepackage{framed}

\usepackage{lineno}

\usepackage{multirow}

\usepackage{footmisc}
\usepackage{todonotes}

\newtheorem{theorem}{Theorem}[section]

\newtheorem{definition}[theorem]{Definition}
\newtheorem{lemma}[theorem]{Lemma}

\newtheorem{question}[theorem]{Question}

\newtheorem{fact}[theorem]{Fact}

\newtheorem{remark}[theorem]{Remark}

\usepackage{stmaryrd}

\begin{document}
\title[]{Approximation Theory and Elementary Submodels}

\author{Sean Cox}
\email{scox9@vcu.edu}
\address{
Department of Mathematics and Statistics \\
Virginia Commonwealth University \\
1015 Floyd Avenue \\
Richmond, Virginia 23284, USA 
}

\date{\today}

\thanks{Partially supported by NSF grant DMS-2154141.  We thank Jan \v{S}aroch and Jan Trlifaj for patient explanations of certain homological arguments.  We thank two anonymous referees for an extensive list of improvements and corrections.}


\begin{abstract}
\emph{Approximation Theory} uses nicely-behaved subcategories to understand entire categories, just as  projective modules are used to approximate arbitrary modules in classical homological algebra.  We use set-theoretic \emph{elementary submodel arguments} to give new, short proofs of well-known theorems in approximation theory, sometimes with stronger results.   
\end{abstract}

\maketitle

\tableofcontents

\section{Introduction}\label{sec_Intro}

Homological algebra is based on the nicely-behaved class of projective modules (direct summands of free modules).  Any module can be approximated by \emph{projective resolutions}, which yield important ring-theoretic and module-theoretic invariants.  If $\mathcal{F}$ is another class of modules, one could attempt homological algebra ``relative to $\mathcal{F}$", with $\mathcal{F}$ playing the role that the projective modules play in the ordinary setting.  This originated with Eilenberg-Moore's 1965 AMS Memoir \cite{MR178036}.  However, for this to be successful, it is essential that $\mathcal{F}$ be a \emph{precovering class}, which ensures that any module can be nicely approximated by modules from $\mathcal{F}$, and that certain homological information is invariant with respect to the choice of approximation.  The notion of a precovering class is due to Enochs, and is closely related to earlier work of Auslander, Salce, and others.  Precovering classes are also called \emph{right-approximating} classes.   For this exposition, we stick mainly to categories of modules, though the topic (\emph{approximation theory}) extends much further, even to some non-abelian categories.  The standard modern references are Enochs-Jenda~\cite{MR1753146} and G\"obel-Trlifaj~\cite{MR2985554}.

Approximation theory really took off around the year 2000, when Bican-El Bashir-Enochs~\cite{MR1832549} proved the \emph{Flat Cover Conjecture}, which asserts that over any ring, the class of flat modules is a covering class (an even stronger property than being a precovering class).  Their proof used a crucial result of Eklof-Trlifaj~\cite{MR1798574} that was later realized to be related to Quillen's \emph{Small Object Argument} from homotopy theory.  The Eklof-Trlifaj argument quickly led to the now-ubiquitous notion of a \textbf{deconstructible class} of modules, which is a sufficient condition for the class to be a precovering class.  

The definition of deconstructibility is a ``bottom-up" definition, asserting that modules in the class can be built up in a certain transfinite manner, from a fixed set of modules in the class (see Section \ref{sec_Decon}).  In \cite{Cox_MaxDecon}, we gave a new ``top-down" characterization of deconstructibility that is reproduced here, in slightly weaker but simpler form, as Theorem \ref{thm_DeconChar} on page \pageref{thm_DeconChar}.  It says that a class $\mathcal{C}$ of $R$-modules is deconstructible if and only if both of the following hold:
\begin{enumerate}[label=(\Alph*)]

 \item\label{item_TheSetTheoryPart} there is a cardinal $\kappa=\kappa_{\mathcal{C}}$ such that \textbf{if} $\mathfrak{N}$ is a ``typical" $\Sigma_1$-elementary submodel of the set-theoretic universe such that $\mathfrak{N} \cap \kappa$ is transitive,\footnote{I.e., viewing $\kappa$ as an ordinal, $\mathfrak{N} \cap \kappa$ is downward closed, so either $\mathfrak{N} \cap \kappa \in \kappa$ or $\kappa \subset \mathfrak{N}$.   For $\kappa = \omega_1$ this is automatic for ``typical" $\mathfrak{N}$, but for $\kappa \ge \omega_2$ more care is involved because of the possibility of \emph{Chang's Conjecture}.  Note that the hypotheses of \ref{item_TheSetTheoryPart} do not place cardinality restrictions on $\mathfrak{N}$, but in certain situations we can restrict attention to $\mathfrak{N}$ of size $<\kappa$ (see Theorem \ref{thm_TurnStatToClub}).} \textbf{then} the following implication holds for any $C$:
\[
C \in \mathfrak{N} \cap \mathcal{C} \implies \left( \mathfrak{N} \cap C \in \mathcal{C} \text{ and } \frac{C}{\mathfrak{N} \cap C} \in \mathcal{C} \right).
\]
 \item\label{item_FiltrationClosedPart} $\mathcal{C}$ is \emph{filtration-closed} (a.k.a.\ ``closed under transfinite extensions", see Section \ref{sec_Decon}).
\end{enumerate}

The real effort-saving part of this characterization is part \ref{item_TheSetTheoryPart}.  For example, we give a new proof of:
\begin{theorem}[\v{S}\v{t}ov\'{\i}\v{c}ek~\cite{MR2476814}]\label{thm_Stovicek_Intro}
If $\left( \mathcal{C}_i \right)_{i \in I}$ is a set-indexed collection of deconstructible classes, then $\bigcap_{i \in I} \mathcal{C}_i$ is deconstructible.
\end{theorem}
\noindent To verify that $\bigcap_{i \in I} \mathcal{C}_i$ satisfies clause \ref{item_TheSetTheoryPart} from the characterization of deconstructibility, let $\kappa_i$ be a cardinal witnessing that \ref{item_TheSetTheoryPart} holds of $\mathcal{C}_i$, and let $\kappa$ be any cardinal larger than $|I|$ and all the $\kappa_i$'s; this is possible because $I$ is a set (not a proper class).  We claim $\kappa$ satisfies clause \ref{item_TheSetTheoryPart} for the class $\bigcap_{i \in I} \mathcal{C}_i$.  One just has to show that \textbf{if} $\mathfrak{N}$ is a ``typical" $\Sigma_1$-elementary submodel of the universe whose intersection with $\kappa$ is transitive and \textbf{if} 
\begin{equation}\label{eq_C_in_N_Demo}
C \in \mathfrak{N} \cap \bigcap_{i \in I} \mathcal{C}_i,
\end{equation}
\textbf{then}
\begin{equation}\label{eq_IntroToVerify}
\mathfrak{N} \cap C \in \bigcap_{i \in I} \mathcal{C}_i \text{ and } \frac{C}{\mathfrak{N} \cap C} \in \bigcap_{i \in I} \mathcal{C}_i.
\end{equation}
But since $\mathfrak{N} \cap \kappa$ is transitive, it follows that $\mathfrak{N}$'s intersection with each $\kappa_i$ is transitive (possibly $\kappa_i$ itself), and hence $\mathfrak{N}$ satisfies clause \ref{item_TheSetTheoryPart} for each $\mathcal{C}_i$.  And assumption \eqref{eq_C_in_N_Demo} ensures that $C \in \mathfrak{N} \cap \mathcal{C}_i$ for each $i$.  So for each $i \in I$, $\mathfrak{N} \cap C$ and $\frac{C}{\mathfrak{N} \cap C}$ are both in $\mathcal{C}_i$.  This verifies \eqref{eq_IntroToVerify}.  We glossed over a few details, mainly in what is meant by ``typical", but those few sentences give the central idea.  See Section \ref{sec_SetSizedIntersect} for the rigorous proof, which is still quite short.\footnote{Clause \ref{item_FiltrationClosedPart} from page \pageref{item_FiltrationClosedPart}  is trivially inherited by $\mathcal{C}$ from the $\mathcal{C}_i$'s here.  Clause \ref{item_FiltrationClosedPart} is typically verified purely via homological algebraic arguments, with less set-theoretic involvement.  Clause \ref{item_FiltrationClosedPart} automatically holds for some of the most widely-studied classes, those of the form ``left half of a cotorsion pair" (a.k.a. ``(left) root of Ext"), because of \emph{Eklof's Lemma} (Section \ref{sec_EklofTrlifaj}).}  

The characterization of deconstructibility also yields a quick proof of:
\begin{theorem}[Bican-El Bashir-Enochs~\cite{MR1832549}]
The class of flat\footnote{A module is flat if tensoring with it preserves exactness of all short exact sequences.} modules is deconstructible.
\end{theorem}
To verify clause \ref{item_TheSetTheoryPart} for the class of flat modules, we let $\kappa:=|R|^+$ where $R$ is the underlying ring.  We want to verify that \textbf{if} $\mathfrak{N}$ is a ``typical" $\Sigma_1$-elementary submodel of the universe whose intersection with $\kappa$ is transitive, and $F \in \mathfrak{N}$ is a flat $R$-module, \textbf{then} both $\mathfrak{N} \cap F$ and $\frac{F}{\mathfrak{N} \cap F}$ are flat.  But one easily checks that since $\mathfrak{N}$ is sufficiently elementary in the universe of sets, then $\mathfrak{N} \cap F$ is a \emph{pure} (in fact, elementary) submodule of $F$; see Lemma \ref{lem_ElemPureReflect}.  Then, by a standard module-theoretic fact, it immediately follows that both $\mathfrak{N} \cap F$ and $\frac{F}{\mathfrak{N} \cap F}$ are flat.  Hence, $\kappa$ witnesses that clause \ref{item_TheSetTheoryPart} holds for the class of flat modules.

This expository paper is about the applications of set-theoretic elementary submodel arguments to approximation theory and relative homological algebra, and particularly applications of the characterization of deconstructibility such as the ones sketched above.  Section \ref{sec_ElemSub} covers basic examples, and in Theorem \ref{thm_DeconChar} gives the precise statement of the top-down characterization of deconstructibility.  Section \ref{sec_ApplicationsBasic} uses this characterization to re-prove many results in the literature about deconstructibility, roots of Ext, and complete cotorsion pairs; the various subsections are largely independent of each other.  Section \ref{sec_TraceComplicated} introduces the use of elementary submodels with complexes of modules, and Section \ref{sec_MoreApps} gives applications, including some new results (e.g. Theorem \ref{thm_GP_ZFC}) and clearer statements of some theorems that were implicit in \cite{Cox_MaxDecon} (e.g., Theorem \ref{thm_TurnStatToClub}).  Section \ref{sec_Salce} discusses independence of Salce's Problem.  Section \ref{sec_Questions} poses some questions.  Appendix \ref{app_ProofDeconChar} provides a proof of Theorem \ref{thm_DeconChar} that is simpler than the original from \cite{Cox_MaxDecon}.

Elementary submodel arguments are ubiquitous in set theory where, for example, they are often used to verify that a forcing poset belongs to Shelah's class of \emph{proper forcings} \cite[Theorem 31.7]{MR1940513}.  Elementary submodel arguments are also well known in topology and infinite combinatorics, where they have been the subject of several nice expositions, e.g., Dow~\cite{MR1031969}, Geschke~\cite{MR1950041}, Just-Weese~\cite{MR1474727}, Soukup~\cite{MR2800978}.  The following quote from Alan Dow's classic exposition of the method (for topology) perfectly matches our goals in this paper (for approximation theory):

\begin{quote}
\itshape `` \dots elementary submodels provide:
\begin{enumerate}
 \item  a convenient shorthand encompassing all standard
closing-off arguments;
 \item  a powerful technical tool which can be avoided
but often at great cost in both elegance and
clarity; and
 \item a powerful conceptual tool providing greater
insight into the structure of the set-theoretic
universe.
\end{enumerate}

I hope to convince some readers of the validity of
these points simply by (over-)using elementary submodels
in proving some new and old familiar results."

\hfill --Dow~\cite{MR1031969}
\end{quote}

\section{Elementary submodels in algebra}\label{sec_ElemSub}

By \textbf{ring} we will mean a (not necessarily commutative) ring with a multiplicative identity 1, though the latter is not essential to many of the arguments.  If $R$ is a ring, \textbf{$\boldsymbol{R}$-module} will denote a left $R$-module, unless otherwise indicated.  If $R$ is the ring $\mathbb{Z}$ of integers, the $R$-modules are just the abelian groups.  An $R$-module homomorphism is an $R$-linear map between $R$-modules.  A \textbf{short exact sequence} of modules is a sequence
\[
\begin{tikzcd}
0 \arrow[r] & A \arrow[r, "\iota"] & B \arrow[r, "\pi"] & C \arrow[r] & 0
\end{tikzcd}
\]
such that the image of each map equals the kernel of the next, in which case $\iota$ is injective, $\pi$ is surjective, and $C \simeq \frac{B}{\text{im} \ \iota}$.

If $R$ is a ring, $M$ is an $R$-module, and $X \subset M$, $\boldsymbol{\langle X \rangle_R}$ denotes the $R$-span of $X$; i.e., the $R$-submodule of $M$ generated by $X$.  We omit the $R$ subscript when it is clear from the context.  If $I$ is a set and $M_i$ is a module for each $i \in I$, $\boldsymbol{\bigoplus_{i \in I} M_i}$ denotes the submodule of $\Pi_{i \in I} M_i$ consisting of those sequences with finite support.  An $R$-module $F$ is \textbf{free} if there is a subset $B \subset F$ that is $R$-linearly independent and spans $F$; such a $B$ is called an $R$-basis of $F$.  If $F$ is free with basis $B$ and $x \in F$, $\boldsymbol{\textbf{sprt}_B(x)}$ will denote the unique minimal finite subset of $B$ such that $x \in \left\langle \text{sprt}_B(x)\right\rangle$.  A module $A$ is a \textbf{direct summand} of $B$ if there exists a module $X$ such that $A \oplus X \simeq B$.  A module is \textbf{projective} if it is a direct summand of a free module.  For some rings such as $\mathbb{Z}$, freeness is equivalent to projectivity, but this is not true in general.

The universe of sets is denoted by $V$.  For a regular uncountable cardinal $\kappa$, $\wp_\kappa(X)$ denotes the set $\{ x \in \wp(X) \ : \ |x|<\kappa\}$, and $\wp^*_\kappa(X)$ denotes the set of $x \in \wp_\kappa(X)$ such that $x \cap \kappa$ is transitive.  A collection $S \subseteq \wp_\kappa(X)$ is \textbf{stationary} (in $\wp_\kappa(X)$) if it meets every closed unbounded subset of $\wp_\kappa(X)$ in the sense of Jech~\cite{MR1940513}; this is equivalent (see \cite[Lemma 3.7]{MattHandbook} or \cite[Lemma 0]{MR924672}) to requiring that for every first order structure $\mathfrak{A}=(X,\dots)$ in a countable signature, there is an $x \in S \cap \wp^*_\kappa(X)$ such that $x \prec \mathfrak{A}$.\footnote{For $\kappa = \omega_1$ it does not matter whether one uses $\wp^*_\kappa(X)$ or $\wp_\kappa(X)$ in the definition of stationarity.  But for $\kappa \ge \omega_2$, it can matter, depending on whether or not \emph{Chang's Conjecture} holds.}  

For a cardinal $\theta$, $H_\theta$ will denote the collection of sets of hereditary cardinality less than $\theta$.  $H_\theta$ is a transitive set, and just as with the cumulative hierarchy (typically denoted by $V_\alpha$'s), every set is a member of some $H_\theta$.  We list some standard facts about $H_\theta$ and elementary submodels of $H_\theta$, all of which can be found in either Kunen~\cite{MR597342}, Jech~\cite{MR1940513}, or other expository papers on the elementary submodel technique (e.g.\ Dow~\cite{MR1031969}, Geschke~\cite{MR1950041}, Soukup~\cite{MR2800978}).  If $\mathfrak{N}=(N,\in) \prec (H_\theta,\in)$,\footnote{I.e., if $(N,\in)$ is an elementary substructure of $(H_\theta,\in)$, where $\in$ is the binary membership relation.} we will typically not distinguish between $\mathfrak{N}$ and its universe $N$.  Readers more comfortable with the cumulative hierarchy can just as well use $V_\theta$'s in place of $H_\theta$'s throughout this survey, with minor adjustments to the arguments.  Formulas in the language of set theory are called $\Delta_0$ ($=\Sigma_0 = \Pi_0$) if all quantifiers are bounded (i.e., either of the form $\forall x \in u$ or $\exists x \in u$), and for $n \ge 1$, a formula is $\Sigma_n$ if it is of the form
\[
\exists x_1 \dots \exists x_k \ \psi(x_1,\dots, x_k, \dots )
\]
for some $\Pi_{n-1}$ formula $\psi$, and a formula is $\Pi_n$ if it is of the form 
\[
\forall x_1 \dots \forall x_k \ \tau(x_1,\dots, x_k, \dots )
\]
for some $\Sigma_{n-1}$ formula $\tau$.
\begin{fact}\label{fact_BasicFactsElemSub}
Suppose $\theta$ is an uncountable cardinal.  Then:
\begin{enumerate}
 \item $(H_\theta,\in)$ is a $\Sigma_1$-elementary submodel of the universe $(V,\in)$ of sets.  If $\theta$ is also regular, then $(H_\theta,\in)$ is a model of ZFC$^-$ (the ZFC axioms with the powerset axiom removed). 
 \item Suppose $\mathfrak{N} \prec (H_\theta,\in)$.
 \begin{enumerate}
 \item\label{item_FiniteSubset} If $b$ is a finite subset of $\mathfrak{N}$, or a finite sequence of elements of $\mathfrak{N}$, then $b$ is an element of $\mathfrak{N}$.
 \item\label{item_CountableElementIsSubset} $\mathfrak{N} \cap \omega_1$ is always transitive, and
\[
\left( x \in \mathfrak{N} \text{ and } x \text{ is countable } \right) \implies \ x \subset \mathfrak{N}.
\] 
 
 \item If $\kappa$ is an infinite cardinal with $\kappa \le \theta$, and $\mathfrak{N} \cap \kappa$ is transitive---i.e., either $\mathfrak{N} \cap \kappa \in \kappa$ or $\kappa \subset \mathfrak{N}$---then for any $x$,
 \[
\left( x \in \mathfrak{N} \text{ and } |x|<\kappa \right) \implies x \subset \mathfrak{N}.
 \]
 
 \end{enumerate}

 \item If $\kappa \le \theta$ are both regular and uncountable, then:
 \begin{enumerate}
  \item  There are stationarily many $\mathfrak{N} \prec (H_\theta,\in)$ of size $<\kappa$ whose intersection with $\kappa$ is transitive; i.e., $\{ \mathfrak{N} \in \wp^*_\kappa(H_\theta) \ : \ \mathfrak{N} \prec (H_\theta,\in) \}$ is stationary.  
 \item  If $\alpha^{<\lambda} < \kappa$ for all $\alpha < \kappa$, then stationarily many of these $\mathfrak{N}$ are also \textbf{$\boldsymbol{<\lambda}$-closed}, meaning that ${}^{<\lambda} \mathfrak{N} \subseteq \mathfrak{N}$, i.e., every $<\lambda$-sized subset of $\mathfrak{N}$ is an \emph{element} of $\mathfrak{N}$.  
 \item For an infinite $\lambda$, the set of $<\lambda$-closed members of $\wp^*_\kappa(H_\theta)$ contains a closed unbounded set if and only if $\lambda = \aleph_0$.
 \end{enumerate}
\end{enumerate}
\end{fact}

\subsection{Basic examples}

Given an $R$-module $A$, we will often deal with submodules of $A$ of the form $\mathfrak{N} \cap A$ for some $\mathfrak{N} \prec (H_\theta,\in)$ such that $A \in \mathfrak{N}$, and $R$ is both an element and subset of $\mathfrak{N}$.  This might appear to be a very special collection of submodules of $A$.  But standard set-theoretic facts tell us that \textbf{almost every} submodule of $A$ is of this form!  More precisely:  fix $A$, and consider any $H_\theta$ with $A \in H_\theta$, and let $\kappa$ be any regular uncountable cardinal larger than $|R|$.  Then 
\begin{align*}
\big\{ B \subset A   \ : \ & B = \mathfrak{N} \cap A \text{ for some } \mathfrak{N} \prec (H_\theta,\in) \text{ such that }  \\
& \{ A,R \} \cup R \subset \mathfrak{N} \text{ and } |\mathfrak{N}|<\kappa  \big\}
\end{align*}
is a \emph{closed and unbounded} subset of $\wp_\kappa(A)$ in the sense of Jech~\cite{MR1940513}.  This follows from the ``Skolemization trick" of \cite[p.\ 913]{MattHandbook} (or \cite[Lemma 2.3]{cox2019adjoining}).

Before covering the applications that combine elementary submodel arguments with the notion of deconstructibility, we present a few lemmas that give the flavor of using elementary submodels in algebra. If $F$ is free and $X$ is a submodule of $F$, it is \emph{not} necessarily true that $X$ is free;\footnote{E.g., if $R=\mathbb{Z}/4\mathbb{Z}$, then $F:=R$ is a free $R$-module, but $X:=2\mathbb{Z}/4 \mathbb{Z}$ is a non-free $R$-submodule of $F$.} and even if both $X$ and $F$ are free, it does not follow that $F/X$ is free.\footnote{E.g., if $R=\mathbb{Z}$, then $2\mathbb{Z}$ is a free $R$-submodule of the free $R$-module $\mathbb{Z}$ (i.e., they are both free abelian groups), but $\mathbb{Z}/2\mathbb{Z}$ is not free.}  But ``almost all" submodules of free modules behave nicely.  The use of elementary submodels in the lemma below is quite an overkill---one could rephrase part \eqref{item_WheneverN} in terms of a filter directly on $\wp(F)$---but thinking in terms of elementary submodels will especially help us later, when dealing with more complicated objects.

\begin{lemma}\label{lem_FreeTrace}
Suppose $R$ is a ring.  The following are equivalent for any $R$-module $F$:
\begin{enumerate}
 \item $F$ is free.
 \item\label{item_WheneverN} Whenever $R,F \in \mathfrak{N} \prec (H_\theta,\in)$, then both $\langle \mathfrak{N} \cap F \rangle$ and $\frac{F}{\langle \mathfrak{N} \cap F \rangle}$ are free.
\end{enumerate}
\end{lemma}
\begin{proof}
The $\Leftarrow$ direction is trivial, since if \ref{item_WheneverN} holds and we take $\mathfrak{N}:=(H_\theta,\in)$, then $\langle \mathfrak{N} \cap F \rangle = F$ and $\frac{F}{\langle \mathfrak{N} \cap F \rangle} = 0$.  Even if one only assumed \ref{item_WheneverN} to hold for, say, countable $\mathfrak{N}$, we'd still get freeness of $F$, since the presence of a single submodule $X$ such that both $X$ and $F/X$ are free implies that $F$ is free (since freeness is closed under extensions).

To prove the $\Rightarrow$ direction, assume $F$ is free.  Then there is a free $R$-basis $B \subset F$ for $F$, and any such $B$ is an element of $H_\theta$.  The statement ``$B$ is an $R$-basis of $F$" is easily downward absolute from $(V,\in)$ to $(H_\theta,\in)$ (it is a $\Sigma_0$ statement in the parameters $R^{<\omega}$, $B^{<\omega}$, and $F$, all of which are elements of $H_\theta$).  So $(H_\theta,\in) \models$ ``there exists a free $R$-basis for $F$".  Because $R, F \in \mathfrak{N} \prec (H_\theta,\in)$, there is such a $B$ with $B \in \mathfrak{N}$.  Trivially, since $B$ is partitioned by $\mathfrak{N} \cap B$ and $B \setminus \mathfrak{N}$,
\begin{equation}
F = \langle \mathfrak{N} \cap B \rangle \oplus \langle B \setminus \mathfrak{N} \rangle,
\end{equation}
which implies that $\frac{F}{\langle \mathfrak{N} \cap B \rangle} \simeq \langle B \setminus \mathfrak{N} \rangle$.  Then
\begin{equation}\label{eq_QuotFree}
\langle \mathfrak{N} \cap B \rangle \text{ and } \frac{F}{\langle \mathfrak{N} \cap B \rangle} \simeq \langle B \setminus \mathfrak{N} \rangle \text{ are free},
\end{equation}
because $B$ is linearly independent (which is obviously inherited by subsets of $B$).

We claim that
\begin{equation}\label{eq_CapB_CapF}
\langle \mathfrak{N} \cap B \rangle = \langle \mathfrak{N} \cap F \rangle,
\end{equation}
which, together with \eqref{eq_QuotFree}, will complete the proof.  For the nontrivial inclusion of \eqref{eq_CapB_CapF}, let $x \in \langle \mathfrak{N} \cap F \rangle$.  Then $x = \sum_{i=1}^k r_i x_i$ for some $x_i \in \mathfrak{N} \cap F$ and $r_i \in R$, but since we are not assuming $R \subset \mathfrak{N}$, some of the $r_i$ might fail to be in $\mathfrak{N}$.  Now
\begin{equation}\label{eq_SprtContainedUnion}
\text{sprt}_B(x) \subseteq \bigcup_{i=1}^k \text{sprt}_B(x_i)
\end{equation}
because $x$ is in the span of $\{ x_1, \dots, x_k \}$.  And for each $i \in \{ 1, \dots, k \}$, $\text{sprt}_B(x_i)$ is an element of $\mathfrak{N}$, because both $B$ and $x_i$ are elements of $\mathfrak{N}$, and because $\text{sprt}_B(x_i)$ is definable over $(H_\theta,\in)$ from the parameters $B$, $R$, and $x_i$.  So, the right side of \eqref{eq_SprtContainedUnion}, which we'll denote by $S$, is a finite union of elements of $\mathfrak{N}$, and hence $S$ is both an element and subset of $\mathfrak{N}$ by Fact \ref{fact_BasicFactsElemSub}.  So the $B$-support of $x$ is contained in $\mathfrak{N} \cap B$, showing that $x \in \langle \mathfrak{N} \cap B \rangle$. 
\end{proof}

Lemma \ref{lem_FreeTrace} is one of the few places in this paper where we do \emph{not} assume that $R$ is a subset of $\mathfrak{N}$; typically we will make that assumption to simplify arguments.

Elementary submodels distribute across quotients, as long as the numerator and denominator are elements of the elementary submodel:
\begin{lemma}\label{lem_IsoQuotient}
Suppose $A$ is an $R$-submodule of $B$ and
\[
\{ A,B,R \} \cup R \subset \mathfrak{N} \prec (H_\theta,\in).
\]
Then $\mathfrak{N} \cap \frac{B}{A} \simeq \frac{\mathfrak{N} \cap B}{\mathfrak{N} \cap A}$.
\end{lemma}
\begin{proof}
Define
\[
\varphi: \mathfrak{N} \cap B \to \mathfrak{N} \cap \frac{B}{A}
\]
by $b \mapsto b + A$; note that the output really is an element of $\mathfrak{N} \cap \frac{B}{A}$, because $\{ b, A,B \} \subset \mathfrak{N} \prec (H_\theta,\in)$, so $\mathfrak{N}$ sees the module $B/A$ and the coset $b + A$.  And $\varphi$ is surjective, because for any coset $z \in \mathfrak{N} \cap \frac{B}{A}$, by elementarity of $\mathfrak{N}$ there is some $b \in \mathfrak{N} \cap B$ such that $z = b + A$.  And clearly $\text{Ker} \ \varphi = \mathfrak{N} \cap A$.  The lemma then follows by the first isomorphism theorem.
\end{proof}

\noindent If $R$ is $\kappa$-Noetherian and $\mathfrak{N} \cap \kappa$ is transitive, one can remove the assumption that $R \subset \mathfrak{N}$ from Lemma \ref{lem_IsoQuotient}, but one has to close each side of the isomorphism in the statement of the lemma under scalar multiplication; see Lemma 3.9 of \cite{Cox_MaxDecon}.  We will not use that generalization here.

An $R$-module $B$ is a \textbf{pure} submodule of $C$ if whenever $M$ is a finite matrix with entries from $R$, $b$ is a column vector from $B$, and $Mx = b$ has a solution in $C$, then it has a solution in $B$.  An embedding $e: B \to C$ is pure if $e[B]$ is pure in $C$.

\begin{lemma}\label{lem_ElemPureReflect}
Suppose $R$ is a ring, $\theta$ is a regular uncountable cardinal, and $\mathfrak{N} \prec (H_\theta,\in)$ is such that $R$ is both an element \textbf{and subset} of $\mathfrak{N}$.  If $A$ is an $R$-module and $A \in \mathfrak{N}$, then $\mathfrak{N} \cap A$ is a pure submodule of $A$.  
\end{lemma}
\begin{proof}
For expository reasons we will prove the lemma directly, though we could derive the stronger fact that $\mathfrak{N} \cap A \prec_{\mathcal{L}_R} A$, where $\mathcal{L}_R$ is the language of $R$-modules.\footnote{$\mathcal{L}_R$ is the language of abelian groups together with, for each $r \in R$, a function symbol interpreted as scalar multiplication by $r$.  The assumption $R \cup \{R \} \subset \mathfrak{N} \prec (H_\theta,\in)$ ensures that the signature of $\mathcal{L}_R$ is both an element and subset of $\mathfrak{N}$.  Then one could prove by recursion on formula complexity that (or use Skolem functions) to conclude that $\mathfrak{N} \cap A \prec_{\mathcal{L}_R} A$ for every $R$-module $A \in \mathfrak{N}$.}

The fact that $R$ is an element and subset of $\mathfrak{N}$, together with elementarity of $\mathfrak{N}$, implies that $\mathfrak{N} \cap A$ is closed under addition and under scalar multiplication from $R$; so $\mathfrak{N} \cap A$ is a submodule of $A$.  To see that it is a pure submodule, suppose $M$ is an $m \times n$ matrix with entries from $R$, $b = [ b_1, \dots, b_m]^T$ is a column vector from $\mathfrak{N} \cap A$, and that $Mx = b$ has a solution in $A$; so
\[
(V,\in) \models \ \underbrace{\exists a_1 \in A \dots \exists a_n \in A \ \   M [a_1 \dots a_n]^T = b}_{\varphi(A,M,R,b)}.
\]
Since $\varphi$ can be expressed a $\Sigma_0$ formula in the language of set theory (in the parameters $A$, $M$, $b$, and $R$), and 
\[
(H_\theta,\in) \prec_{\Sigma_1} (V,\in),
\]
it follows that
\[
(H_\theta,\in) \models \varphi(A,M,R,b).
\]
By elementarity of $\mathfrak{N}$, and the fact that $M$ is finite and $R \subset \mathfrak{N} \prec H_\theta$, it follows from Fact \ref{fact_BasicFactsElemSub} part \eqref{item_FiniteSubset} that $M$ is an \emph{element} of $\mathfrak{N}$.  Also, $b$ is a finite sequence of elements of $\mathfrak{N}$, and hence, again by Fact \ref{fact_BasicFactsElemSub} part \eqref{item_FiniteSubset}, $b$ is an element of $\mathfrak{N}$.  So all free parameters in the formula $\varphi(A,M,R,b)$ are elements of $\mathfrak{N}$.  So by elementarity,
\[
\mathfrak{N} \models \varphi(A,M,R,b).
\]
Let $a_1, \dots, a_n \in \mathfrak{N} \cap A$ witness the existential quantifiers of $\varphi$.  Then $\vec{a}$ is a finite sequence of elements of $\mathfrak{N} \cap A$ that solve $Mx=b$.
\end{proof}

\begin{remark}
One could also prove Lemma \ref{lem_ElemPureReflect} using the characterization of purity as those embeddings whose associated short exact sequences remain exact after applying $\text{Hom}(A,-)$, for any finitely presented module $A$.\footnote{The key point is that if $R \cup \{R \} \subset \mathfrak{N} \prec (H_\theta,\in)$, then every finitely presented $R$-module has a representative that is an \emph{element} (and subset) of $\mathfrak{N}$.  Moreover, if $A,Y \in \mathfrak{N}$ with $A$ finitely presented, then $\mathfrak{N} \cap \text{Hom}^V(A,Y) = \text{Hom}^V(A, \mathfrak{N} \cap Y)$ (this equality can fail if $A$ is not finitely presented). }  Although this proof would have more easily generalized to other abelian categories, for the purposes of exposition we opted for the simpler proof of Lemma \ref{lem_ElemPureReflect}  given above, which is specific to module categories.
\end{remark}

Pure embeddings generalize to $\lambda$-pure embeddings for any infinite regular $\lambda$.  An inclusion $B \to C$ is $\lambda$-pure if for any system $\mathcal{S}$ of fewer than $\lambda$ many $R$-linear equations with parameters from $B$:  if there is a solution for the entire system in $C$, then there is a solution in $B$ too.  This is equivalent to requiring that whenever $M$ is a matrix of entries from $R$ with $<\lambda$-many rows and (wlog) $<\lambda$-many columns (and each row finitely supported), $b$ is a vector in $B$, and $Mx=b$ has a solution in $C$, then it has a solution in $B$.  
 
\begin{remark}\label{rem_LambdaPureLambdaclosed}
If the $\mathfrak{N}$ from the statement of Lemma \ref{lem_ElemPureReflect} is assumed to be closed under sequences of length less than some regular cardinal $\lambda$, then the conclusions of Lemma \ref{lem_ElemPureReflect} can be strengthened by replacing every instance of ``pure" by ``$\lambda$-pure".   
\end{remark}

\subsection{Deconstructible classes}\label{sec_Decon}

Given a class $\mathcal{C}$ of $R$-modules, a \textbf{$\boldsymbol{\mathcal{C}}$-filtration} is a $\subseteq$-increasing and $\subseteq$-continuous\footnote{I.e., if $i$ is a limit ordinal then $M_i = \bigcup_{k < i} M_k$.} chain of $R$-modules
\[
\langle M_i \ : \ i < \zeta \rangle
\]
such that:
\begin{enumerate}
 \item $\zeta$ is an ordinal;
 \item $M_0 =0$;\footnote{Sometimes we allow $M_0 \in \mathcal{C}$.  In light of requirement \eqref{item_QuotRequirement}, this makes no difference to membership in $\text{Filt}(\mathcal{C})$ as defined below.} and
 \item\label{item_QuotRequirement} for each $i$ such that $i+1 < \zeta$, $M_{i+1}/M_i$ is isomorphic to some member of $\mathcal{C}$.  Sometimes we will simply write ``$M_{i+1}/M_i \in \mathcal{C}$" when we really mean it is isomorphic to some member of $\mathcal{C}$.
\end{enumerate}
A module $M$ is said to be \textbf{$\mathcal{C}$-filtered} if it is the union of some $\mathcal{C}$-filtration (note that if it is the union of a filtration of successor length, then $M$ is simply the last module in the $\mathcal{C}$-filtration).  The class of $\mathcal{C}$-filtered modules is denoted $\textbf{Filt}\boldsymbol{(\mathcal{C})}$.  Clearly $\mathcal{C} \subseteq \text{Filt}(\mathcal{C})$ because each $C \in \mathcal{C}$ can be viewed as the union of the $\mathcal{C}$-filtration $\langle 0, C \rangle$ of length 2.  We say that $\mathcal{C}$ is \textbf{filtration-closed} (or \textbf{closed under transfinite extensions}) if the reverse inclusion holds too; i.e., if and only if $\mathcal{C} = \text{Filt}(\mathcal{C})$.  For any class $\mathcal{C}$, $\text{Filt}(\mathcal{C})$ is closed under filtrations; i.e.
\begin{equation}\label{eq_FiltSquared}
\text{Filt}(\text{Filt}(\mathcal{C})) = \text{Filt}(\mathcal{C}). \tag{$\star$}
\end{equation}
To see the nontrivial $\subseteq$ direction, suppose $M$ is the union of some continuous chain $\langle M_i \ : \ i \le \zeta \rangle$ with each $M_{i+1}/M_i \in \text{Filt}(\mathcal{C})$; so $M_{i+1}/M_i$ is the union of some $\mathcal{C}$-filtration $\vec{Z}^i$.  Then by the Correspondence Theorem, we can view $\vec{Z}^i$ as a smooth chain of the form
\[
\left\langle N^i_\xi/M_i \ : \ \xi < \zeta^i \right\rangle
\]
where for each $\xi$ (such that $\xi +1 < \zeta^i$), 
\[
\frac{N^i_{\xi+1}/M_i}{N^i_\xi/M_i} \in \mathcal{C}.
\] 
The quotient above is isomorphic to $N^i_{\xi+1}/N^i_\xi$; so $N^i_{\xi+1}/N^i_\xi \in \mathcal{C}$.  Concatenating all the $\vec{N}^i$'s together yields a $\mathcal{C}$-filtration with the same union as the original filtration $\vec{M}$.

A class $\mathcal{C}$ is called \textbf{deconstructible} if there is a subset $\mathcal{C}_0 \subseteq \mathcal{C}$---where $\mathcal{C}_0$ is a set, not a proper class!---such that
\begin{equation}\label{eq_EqualsFilt}
\mathcal{C} = \text{Filt}(\mathcal{C}_0).
\end{equation}
This definition agrees with, for example, \v{S}\v{t}ov\'{\i}\v{c}ek~\cite{MR3010854}, Saor\'{\i}n-\v{S}\v{t}ov\'{\i}\v{c}ek~\cite{MR2822215}, \v{S}aroch-Trlifaj~\cite{MR4835280}, and most of the newer papers on the topic.  In other sources, such as G\"obel-Trlifaj~\cite{MR2985554} and the author's own \cite{Cox_MaxDecon}, deconstructiblity was defined by replacing the ``$=$" in \eqref{eq_EqualsFilt} by ``$\subseteq$".  We will refer to this notion as \textbf{weakly deconstructible}.  For many natural classes, e.g.\ the ``left roots of Ext", the notions are equivalent (by Eklof's Lemma \ref{lem_EklofLemma}).  Using \eqref{eq_FiltSquared}, it is straightforward to prove:
\begin{lemma}\label{lem_FiltOfWeakDecon}
The following are equivalent for any class $\mathcal{C}$:
\begin{enumerate}
 \item $\mathcal{C}$ is deconstructible.
 \item $\mathcal{C}$ is weakly deconstructible and filtration-closed.
\end{enumerate}
\end{lemma}

Deconstructibility of $\mathcal{C}$ is equivalent to requiring the existence of a cardinal $\kappa$ such that
\[
\mathcal{C} = \text{Filt}\left( \{ C \in \mathcal{C} \ : \ |C|<\kappa  \} \right),
\]
and it will be convenient to have a cardinal-specific notion of deconstructibility.  Following the prevailing terminology, we will say that $\mathcal{C}$ is \textbf{$\boldsymbol{\kappa}$-deconstructible} if 
\[
\mathcal{C} = \text{Filt}\left( \underbrace{\{C \in \mathcal{C} \ : \ C \text{ is } <\kappa \text{-presented}   \}}_{=:\mathcal{C}_\kappa} \right),
\]
where ``$C$ is $<\kappa$-presented" means that $C=F/K$ for some free $F$, where both $F$ and $K$ are generated (as $R$-modules) by sets of size less than $\kappa$.  Observe that $\mathcal{C}$ is deconstructible if and only if it is $\kappa$-deconstructible for some $\kappa$. Similarly, we will say that $\mathcal{C}$ is \textbf{weakly $\boldsymbol{\kappa}$-deconstructible} if 
\[
\mathcal{C} \subseteq \text{Filt}\left( \mathcal{C}_\kappa \right),
\]
Clearly, if $\kappa_0 < \kappa_1$, then $\kappa_0$-(weak)-deconstructibility implies $\kappa_1$-(weak)-deconstructibility.

\begin{remark}\label{rem_R_less_kappa}
For this article, we will frequently assume for simplicity that $|R|<\kappa$, which ensures (assuming $\kappa$ is infinite) that for any $R$-module $C$, the following are equivalent:
\begin{itemize}
 \item $C$ is $<\kappa$-presented
 \item $C$ is $<\kappa$-generated
 \item $|C|<\kappa$.
\end{itemize}
\end{remark}

Recall from the introduction that the precovering property of a class is important for relative homological algebra.  A class $\mathcal{C}$ is a \textbf{precovering class} if for every module $M$, there exists some (not necessarily unique) homomorphism $\pi: C \to M$ with $C \in \mathcal{C}$, such that every morphism from any member $C'$ of $\mathcal{C}$ into $M$ factors through $\pi$:\footnote{In category-theoretic language, this is asserting that $\mathcal{C}$ is a \emph{weakly coreflective class} in $R$-Mod.} 
\[
\begin{tikzcd}
C \arrow[r,"\pi"] & M \\
C' \arrow[ur] \arrow[u, dotted] & 
\end{tikzcd}
\]
Such a map $\pi: C \to M$ is called a $\mathcal{C}$-precover of $M$.

The key fact about deconstructibility is:
\begin{theorem}[Saor\'{\i}n-\v{S}\v{t}ov\'{\i}\v{c}ek~\cite{MR2822215}]\label{thm_DeconPrecover}
All deconstructible classes are precovering classes.  (They proved this is true in many ``exact" categories, not just categories of modules).
\end{theorem}

The proof of Theorem \ref{thm_DeconPrecover} goes through Quillen's \emph{Small Object Argument} from homotopy theory.    Theorem \ref{thm_DeconPrecover} strengthened the main result of Eklof-Trlifaj~\cite{MR1798574}, who had isolated an important special case (for left classes of \emph{cotorsion pairs}) that wasn't as obviously connected to the Small Object Argument, but was crucial to the solution of the Flat Cover Conjecture.

There exist filtration-closed classes that are not even precovering.  A series of papers (Estrada et al.~\cite{MR2964610}, Sl\'{a}vik-Trlifaj~\cite{MR3161764}, Bazzoni-\v{S}\v{t}ov\'{\i}\v{c}ek~\cite{MR2869137}, \v{S}aroch~\cite{MR3819707}) culminated in showing that the class of $\aleph_1$-projective modules---which is the same as the Flat Mittag-Leffler modules (by Herbera-Trlifaj~\cite{MR2900444})---is a precovering class if and only if $R$ is a right perfect ring.  In particular, since $\mathbb{Z}$ is not right perfect, the class of $\aleph_1$-free abelian groups is closed under filtrations, but is not a precovering class.

To head off any possible confusion, we mention a much stronger (and much rarer) property than deconstructibility.  A class $\mathcal{C}$ is called \textbf{decomposable} if there is a $\kappa$ such that every member of $\mathcal{C}$ is a direct sum of $<\kappa$-presented members of $\mathcal{C}$.  Decomposable classes are (weakly) deconstructible, since if 
\[
C = \bigoplus_{i < \lambda} C_i
\]
with each $C_i$ a $<\kappa$-presented member of $\mathcal{C}$, we can view $C$ as the union of the smooth chain 
\[
\left\langle  \underbrace{\bigoplus_{\xi < i} C_\xi}_{=:D_i}  \ : \  i < \lambda \right\rangle,
\]
where $D_{i+1}/D_i \simeq C_i$.  So $\mathcal{C}$ is (weakly) $\kappa$-deconstructible.

Kaplansky's Theorem (\cite{MR0100017}) states that over every ring, the class of projective modules is decomposable (with $\kappa = \aleph_1$).  But deconstructibility is strictly weaker than decomposability; e.g., the class of flat modules is always deconstructible (see Section \ref{sec_FCC}), but for many rings the class of flat modules is \emph{not} decomposable.

\subsection{Top-down characterization of deconstructibility}

In \cite{Cox_MaxDecon} we provided a new characterization of  deconstructibility of a class of modules, restated in a simpler (but slightly weaker form) as Theorem \ref{thm_DeconChar} below.  The reader is advised to just focus on clauses \ref{item_KappaDecon} and \ref{item_MainChar} for now; clauses \ref{item_DiagVersion2models} and \ref{item_non_diag} are technical variants that are only occasionally used.

\begin{theorem}[Cox~\cite{Cox_MaxDecon}, Theorem 1.1]\label{thm_DeconChar}
Suppose $R$ is a ring and $\kappa$ is a regular uncountable cardinal with $\kappa > |R|$.\footnote{Theorem 1.1 of \cite{Cox_MaxDecon} used the weaker assumption that $R$ was a $\kappa$-Noetherian ring.  We chose here to make the assumption that $|R|<\kappa$, in order to simplify the statement and proof.  If one is only interested in whether a class is deconstructible or not---without caring how large the filtering set is---then the theorem as stated here carries the full power of Theorem 1.1 of \cite{Cox_MaxDecon}.}  For any class $\mathcal{C}$ of $R$-modules that is closed under filtrations, the following are equivalent:
\begin{enumerate}[label=(\Roman*)]
 \item\label{item_KappaDecon} $\mathcal{C}$ is $\kappa$-deconstructible.
 \item\label{item_MainChar} (``Diagonal" version, single model) For all sufficiently large regular cardinals $\theta$, whenever 
 \[
 \mathfrak{N} \prec (H_\theta,\in, R, \kappa, \mathcal{C} \cap H_\theta)
 \]
 and $\mathfrak{N} \cap \kappa$ is transitive, then the following implication holds for all $R$-modules $C$:
  \[
  C \in \mathfrak{N} \cap \mathcal{C} \ \implies \ \left( \mathfrak{N} \cap C \in \mathcal{C} \text{ and } \frac{C}{\mathfrak{N} \cap C}  \in \mathcal{C} \right).
  \]
  
 \item\label{item_DiagVersion2models} (``Diagonal" version, two models) For all sufficiently large regular cardinals $\theta$, whenever 
 \[
 \mathfrak{N} \prec \mathfrak{N}' \prec (H_\theta,\in, R, \kappa, \mathcal{C} \cap H_\theta)
 \]
 and both $\mathfrak{N} \cap \kappa$ and $\mathfrak{N}' \cap \kappa$ are transitive, then the following implication holds for all $R$-modules $C$:
  \[
  C \in \mathfrak{N} \cap \mathcal{C} \ \implies \ \left( \mathfrak{N} \cap C \in \mathcal{C} \text{ and } \frac{\mathfrak{N}' \cap C}{\mathfrak{N} \cap C}  \in \mathcal{C}  \right).
  \]

  \item\label{item_non_diag} (``Non-diagonal" version, cf.\ Lemma 2.3 of \cite{Cox_MaxDecon}) For every $C \in \mathcal{C}$ there exists a pair $(H,\mathfrak{A})$ such that $(H,\in)$ is a transitive ZFC$^-$ model, $\{ \kappa, R,  C \} \subset H$, and $\mathfrak{A}$ is an expansion of $(H,\in, C)$ in a countable first order signature such that whenever $\mathfrak{N} \prec \mathfrak{A}$ and $\mathfrak{N} \cap \kappa$ is transitive, then $\mathfrak{N} \cap C \in \mathcal{C}$ and $\frac{C}{\mathfrak{N} \cap C} \in \mathcal{C}$.  
\end{enumerate}

\end{theorem}

The \ref{item_MainChar} $\implies$ \ref{item_KappaDecon} direction of Theorem \ref{thm_DeconChar} provides a ``top-down" way to verify that a class is deconstructible, allowing one to avoid having to recursively build filtrations every time.  The \ref{item_KappaDecon} $\implies$ \ref{item_MainChar} direction of Theorem \ref{thm_DeconChar} is a variant of the Hill Lemma (\cite{MR2985554}, Chapter 7).

\begin{remark}\label{rem_DontNeedFiltClosure}
Theorem \ref{thm_DeconChar} assumes closure of $\mathcal{C}$ under filtrations.  But even without that assumption, one still has the implication
\[
\ref{item_non_diag}  \implies \mathcal{C} \text{ is } \textbf{weakly} \ \kappa \text{-deconstructible}.
\]
\end{remark}

We want to focus mainly on applications of Theorem \ref{thm_DeconChar}, so we relegate its proof to Appendix \ref{app_ProofDeconChar}, and only deal with equivalence of \ref{item_KappaDecon} with \ref{item_MainChar} (see \cite{Cox_MaxDecon} for the proof of further equivalence with clauses \ref{item_DiagVersion2models} and \ref{item_non_diag}, and for how to deal with situations where $\kappa \le |R|$).

\begin{remark}\label{rem_FiltGames}
Though we will not discuss this variant in the current article, we note that Cox~\cite{Cox_DBgames}, building on previous work of Mekler-Shelah-V\"{a}\"{a}n\"{a}nen~\cite{MR1191613}, investigates what happens when one weakens clause \ref{item_DiagVersion2models} to only require the relevant implication to hold when both $\mathfrak{N}$ and $\mathfrak{N}'$ have cardinality $<\kappa$. Roughly speaking, that variant of \ref{item_DiagVersion2models} is equivalent to each of the following statements, which can be viewed as ``generic" versions of asserting that $\mathcal{C}$ is $\kappa$-deconstructible:
\begin{itemize} 
  \item Every module in $\mathcal{C}$ is $\mathcal{C}_{\kappa}$-filtered in some $\kappa$-closed forcing extension of the universe
  \item Player II has a winning strategy in various games of length $\kappa$, including a certain Ehrenfeucht-Fra\"{\i}ss\'{e} game of length $\kappa$ and \emph{Filtration Games} introduced in \cite{Cox_DBgames}.
\end{itemize}
\end{remark}

In Theorem \ref{thm_DeconChar}, it is possible to weaken the assumption that $|R|<\kappa$ to merely assume $R$ is $\kappa$-Noetherian (all ideals in $R$ are $<\kappa$-generated), if one also replaces instances of $\mathfrak{N} \cap C$ in the theorem with $\langle \mathfrak{N} \cap C \rangle$.  This is how it was done in Theorem 1.1 of  \cite{Cox_MaxDecon}, though we avoid doing so here because it considerably complicates the arguments.  Though we will not use it here, the $\kappa$-Noetherian property of a ring can even be characterized in terms of elementary submodels:

\begin{theorem}[Cox~\cite{Cox_MaxDecon}, Theorem 3.8]\label{thm_CharNoeth}
For regular uncountable $\kappa$, the following are equivalent for any ring $R$:
\begin{enumerate}
 \item\label{item_R_Kappa_Noeth} $R$ is $\kappa$-Noetherian.
 
 \item\label{item_N_cap_generated} Whenever $A$ is an $R$-submodule of $M$, $\{ A,M,R \} \subset \mathfrak{N} \prec (H_\theta,\in)$, and $\mathfrak{N} \cap \kappa$ is transitive, then
 \[
\langle \mathfrak{N} \cap M \rangle \cap A = \langle \mathfrak{N} \cap A \rangle.
 \]
 (The $\supseteq$ direction always holds; the $\kappa$-Noetherian property is needed to get the $\subseteq$ direction).
\end{enumerate}
\end{theorem}

\section{Applications:  streamlining older results}\label{sec_ApplicationsBasic}

We use the characterization of deconstructibility (Theorem \ref{thm_DeconChar}) to give short proofs of some known results.

\subsection{Proof of the Flat Cover Conjecture}\label{sec_FCC}

The notion of deconstructibility arose out of the following resolution of the Flat Cover Conjecture around the year 2000.  We give a new proof of:

\begin{theorem}[Eklof-Trlifaj~\cite{MR1798574}, Bican-El Bashir-Enochs~\cite{MR1832549}]\label{thm_FlatCover}
Over any ring $R$, the class of flat $R$-modules is a covering class.
\end{theorem}
Long before the proof of Theorem \ref{thm_FlatCover}, Enochs had reduced the problem to proving that the class $\mathcal{F}$ of flat modules is a \emph{pre}covering class.  And since $\mathcal{F}$ is closed under extensions and direct limits, it is in particular filtration-closed.  So by Theorem \ref{thm_DeconPrecover} it suffices to show that $\mathcal{F}$ is weakly deconstructible.  We will use:
\begin{fact}[Lam~\cite{MR1653294} Corollary 4.86]\label{fact_FlatPure}
If $F$ is flat and $K$ is a pure submodule of $F$, then both $K$ and $F/K$ are flat.
\end{fact}

Let $F$ be a flat module, and suppose $F \in \mathfrak{N} \prec (H_\theta,\in, R)$ with $\mathfrak{N} \cap |R|^+$ transitive; then $R$ is both an element and subset of $\mathfrak{N}$, by Fact \ref{fact_BasicFactsElemSub}.  By Lemma \ref{lem_ElemPureReflect}, $\mathfrak{N} \cap F$ is a pure submodule of $F$.  Then by Fact \ref{fact_FlatPure}, both $\mathfrak{N} \cap F$ and $\frac{F}{\mathfrak{N} \cap F}$ are flat.  By the \ref{item_MainChar} $\implies$ \ref{item_KappaDecon} direction of Theorem \ref{thm_DeconChar}, $\mathcal{F}$ is $|R|^+$-deconstructible, completing the proof of Theorem \ref{thm_FlatCover}.

The argument actually proves the following generalization:
\begin{theorem}[Bravo-Gillespie-Hovey~\cite{bravo2014stable}]\label{thm_Bravo_et_al}
Suppose $\mathcal{C}$ is a class of modules closed under filtrations, and such that for any $C \in \mathcal{C}$ and any submodule $K \subset C$:
\[
K \text{ is pure in } C \ \implies \ \left( K \in \mathcal{C} \text{ and } C/K \in \mathcal{C} \right).
\]  
Then $\mathcal{C}$ is deconstructible.
\end{theorem}

The assumptions on $\mathcal{C}$ in Theorem \ref{thm_Bravo_et_al} closely resemble, but are strictly stronger than, clause \ref{item_MainChar} of Theorem \ref{thm_DeconChar}.  For example, there exist rings for which the projective modules are \textbf{not} closed under pure submodules, so in particular do not satisfy the assumptions of Theorem \ref{thm_Bravo_et_al}.\footnote{The simplest example of such a ring (described to the author by J.\ \v{S}aroch) is the ring $\prod_{n =1}^\infty \mathbb{Z}/2\mathbb{Z}$.  This goes back to Osofsky~\cite{MR0231857}.}  But the class of projective modules is always deconstructible, over any ring,\footnote{Due to Kaplansky's Theorem, see Section \ref{sec_Kaplansky} below.} and hence always satisfy clause \ref{item_MainChar} of Theorem \ref{thm_DeconChar}.

\subsection{The Eklof-Trlifaj extensions of the Flat Cover Conjecture}\label{sec_EklofTrlifaj}

For a class $\mathcal{B}$ of $R$-modules, ${}^\perp \mathcal{B}$ refers to the class of $A$ such that $\text{Ext}^1_R(A,B)=0$.  We will use several standard facts about \emph{Ext} from homological algebra (see, for example, Chapter 1 of Eklof-Mekler~\cite{MR1914985}):
\begin{fact}\label{fact_CharactVanishExt}
$\text{Ext}^1_R(A,B)=0$ if and only if the following holds:  given some (equivalently, any) presentation $A \simeq F/K$ with $F$ projective, every homomorphism from $K$ into $B$ extends to a homomorphism from $F$ into $B$.
\end{fact}

\begin{fact}\label{fact_long_exact}
Suppose
\[
\begin{tikzcd}
0 \arrow[r] & K \arrow[r, "e"] &  M \arrow[r, "\pi"] &  Q \arrow[r] & 0
\end{tikzcd}
\]
is a short exact sequence of $R$-modules, and $B$ is an $R$-module.  Then there is an exact sequence
\[
\begin{tikzcd}
\text{Hom}_R(K,B) \arrow[drr] & \text{Hom}_R(M,B) \arrow[l, "e^*"'] & \text{Hom}_R(Q,B) \arrow[l, "\pi^*"'] & 0 \arrow[l] \\
\text{Ext}^1_R(K,B) \arrow[drr] & \text{Ext}^1_R(M,B) \arrow[l] & \text{Ext}^1_R(Q,B) \arrow[l]  & \\
\text{Ext}^2_R(K,B) & \text{Ext}^2_R(M,B) \arrow[l] & \text{Ext}^2_R(Q,B) \arrow[l]  &  \\
 \vdots &   \vdots &  \vdots & 
\end{tikzcd}
\]
of abelian groups, where $e^*$ and $\pi^*$ are the results of applying the contravariant functor $\text{Hom}(-,B)$ to $e$ and $\pi$, respectively.
\end{fact}

We will exclusively focus on $\text{Ext}^1$, so will omit the superscript from now on.  Given a class $\mathcal{X}$ of modules, 
\[
{}^\perp \mathcal{X} := \bigcap_{X \in \mathcal{X}} \text{Ker} \ \text{Ext}(-,X)= \{ A \ : \ \forall X \in \mathcal{X} \  \text{Ext}(A,X)=0 \}
\]
and
\[
\mathcal{X}^\perp := \bigcap_{X \in \mathcal{X}} \text{Ker} \ \text{Ext}(X,-) = \{ B \ : \ \forall X \in \mathcal{X} \ \text{Ext}(X,B)=0 \}.
\]
A pair of classes $(\mathcal{A},\mathcal{B})$ is a \textbf{cotorsion pair} if $\mathcal{A}^\perp = \mathcal{B}$ and $\mathcal{A} = {}^\perp \mathcal{B}$.

\begin{lemma}[Eklof's Lemma, see \cite{MR2985554}, Lemma 6.2]\label{lem_EklofLemma}
For any class $\mathcal{B}$, ${}^\perp \mathcal{B}$ is filtration closed.  Recall from Section \ref{sec_Decon} this means that
\[
{}^\perp \mathcal{B} = \text{Filt}\left( {}^\perp \mathcal{B} \right).
\]
\end{lemma}

Eklof and Trlifaj generalized Theorem \ref{thm_FlatCover} as follows.  A module $C$ is \textbf{pure-injective} if whenever $A$ is a pure submodule of $B$, then every morphism from $A \to C$ extends to a morphism from $B \to C$ (i.e., $C$ is ``injective with respect to all pure embeddings").  Their generalization states:\footnote{The class of all flat modules is equal to ${}^\perp \mathcal{B}$, where $\mathcal{B}$ is the class of all pure-injective modules (\cite{MR1778163}).  So the proof of Theorem \ref{thm_FlatCover} can be rephrased as:  ${}^\perp \mathcal{B}$ is deconstructible, where $\mathcal{B}$ is the class of all pure-injectives.  Theorem \ref{thm_EklofTrlifaj} generalized that result by replacing $\mathcal{B}$ with \emph{any} class of pure-injective modules.}

\begin{theorem}
[Eklof-Trlifaj~\cite{MR1778163}]\label{thm_EklofTrlifaj}
Suppose $R$ is a ring and $\mathcal{B}$ is a class of pure-injective $R$-modules.  Then ${}^\perp \mathcal{B}$ is deconstructible.
\end{theorem}

\begin{remark}\label{rem_WeakFormPureInj}
The proof below shows that we can weaken the assumption on $\mathcal{B}$, to only require members of $\mathcal{B}$ to be injective with respect to very special kinds of pure embeddings:  those inclusions of the form $\mathfrak{N} \cap A \to A$ where $A \in \mathfrak{N}$ and $R \cup \{R \} \subset \mathfrak{N}$ (i.e., the kind from the statement of Lemma \ref{lem_ElemPureReflect}).   We do not know if this is a strictly weaker property than being pure injective; see Question \ref{q_STT_inj}.
\end{remark}

\begin{proof}
(of Theorem \ref{thm_EklofTrlifaj}) By Eklof's Lemma \ref{lem_EklofLemma}, ${}^\perp \mathcal{B}$ is closed under transfinite extensions, so it suffices to prove that it is weakly deconstructible.  Let $\kappa$ be any regular uncountable cardinal larger than $R$, and consider any $\mathfrak{N} \prec (H_\theta,\in,R,\kappa)$ with $\mathfrak{N} \cap \kappa$ transitive.  Then $R$ is both an element and subset of $\mathfrak{N}$ by Fact \ref{fact_BasicFactsElemSub}.  Consider any $A \in \mathfrak{N} \cap {}^\perp \mathcal{B}$; by Theorem \ref{thm_DeconChar}, it suffices to prove that:
\begin{enumerate}[label=(\roman*)]
  \item\label{item_DownReflect} $\mathfrak{N} \cap A \in {}^\perp \mathcal{B}$; and
  \item\label{item_QuotientPerpB} $\frac{A}{\mathfrak{N} \cap A} \in {}^\perp \mathcal{B}$.
\end{enumerate}

To prove part \ref{item_DownReflect}:  since $A \in \mathfrak{N} \prec (H_\theta,\in,R)$, there is a free module $F$ and a submodule $K$ such that $F,K \in \mathfrak{N}$ and $A \simeq F/K$.  Then by Lemma \ref{lem_IsoQuotient},
\[
\mathfrak{N} \cap A  \simeq \mathfrak{N} \cap \frac{F}{K} \simeq \frac{\mathfrak{N} \cap F}{\mathfrak{N} \cap K}.
\]
Since $F$ is free and $F \in \mathfrak{N}$, $\mathfrak{N} \cap F$ is free by Lemma \ref{lem_FreeTrace}.  So, in order to prove that $\frac{\mathfrak{N} \cap F}{\mathfrak{N} \cap K}$ is in ${}^\perp \mathcal{B}$, it suffices by Fact \ref{fact_CharactVanishExt} to prove that for every $B \in \mathcal{B}$:  every morphism $\varphi: \mathfrak{N} \cap K \to B$ extends to domain $\mathfrak{N} \cap F$.  But by Lemma \ref{lem_ElemPureReflect}, $\mathfrak{N} \cap K$ is a pure submodule of $K$. Hence, since $B$ is pure-injective, $\varphi$ extends to a homomorphism $\varphi_1: K \to B$.   Since $A=F/K \in {}^\perp \mathcal{B}$ by assumption, Fact \ref{fact_CharactVanishExt} ensures that $\varphi_1$ extends to a homomorphism $\varphi_2: F \to B$.  So by restricting $\varphi_2$ to $\mathfrak{N} \cap F$, we have the desired morphism from $\mathfrak{N} \cap F$ into $B$ extending $\varphi$.

To prove part \ref{item_QuotientPerpB}, fix any $B \in \mathcal{B}$, and consider the short exact sequence 
\begin{equation}\label{eq_SES_EklofTrlifaj}
\xymatrix{
0 \ar[r] & \mathfrak{N} \cap A \ar[r]^-i & A \ar[r] & \frac{A}{\mathfrak{N} \cap A} \ar[r] & 0
}
\end{equation}
where $i$ is inclusion, and the associated long exact sequence (cf. Fact \ref{fact_long_exact}) gotten by applying $\text{Hom}(-,B)$ to the short exact sequence \eqref{eq_SES_EklofTrlifaj}:
\begin{equation}\label{eq_AfterApplyHom}
\xymatrix{
\text{Hom}\left( \mathfrak{N} \cap A, B \right) \ar[drr] \ar@{<-}[r]^-{i^*} & \text{Hom}(A,B) \ar@{<-}[r] & \text{Hom}\left( \frac{A}{\mathfrak{N} \cap A}, B \right) \\
 & \underbrace{\text{Ext}(A,B)}_{=0 \text{ because } A\in {}^\perp \mathcal{B}} \ar@{<-}[r] & \text{Ext}\left( \frac{A}{\mathfrak{N} \cap A}, B \right)
}
\end{equation}
The map $i^*: \text{Hom}(A,B) \to \text{Hom}(\mathfrak{N} \cap A ,B)$ is the restriction map $\varphi \mapsto \varphi \restriction (\mathfrak{N} \cap A)$.  By Lemma \ref{lem_ElemPureReflect}, $\mathfrak{N} \cap A$ is a pure submodule of $A$, and since $B$ is pure-injective, it follows that the map $i^*$ is a surjection.   Then, since $\text{Ext}(A,B)=0$, exactness of \eqref{eq_AfterApplyHom} at the lower right entry implies that $\text{Ext} \left( \frac{A}{\mathfrak{N} \cap A},B \right)$ is zero.
\end{proof}

A close reading of the proof, together with Remark \ref{rem_LambdaPureLambdaclosed} used in place of Lemma \ref{lem_ElemPureReflect}, yields the next lemma, which will be used later on.  Recall the definition of a $\lambda$-pure embedding, just before Remark \ref{rem_LambdaPureLambdaclosed}.   A module $B$ is \textbf{$\boldsymbol{\lambda}$-pure-injective} if it is injective with respect to all $\lambda$-pure embeddings; i.e., whenever $X$ is a $\lambda$-pure submodule of $Y$, then every morphism from $X \to C$ extends to a morphism from $Y \to C$.

\begin{lemma}\label{lem_LambdaClosedLeftPerp}
Suppose $R$ is a ring, $\lambda$ is an infinite regular cardinal, and $B$ is a $\lambda$-pure-injective $R$-module.  Then for any $\mathfrak{N} \prec (H_\theta,\in)$ that is $<\lambda$-closed,\footnote{For $\lambda = \aleph_0$ this is superfluous, by Fact \ref{fact_BasicFactsElemSub} part \eqref{item_FiniteSubset}.} the following implication holds for any $A \in \mathfrak{N}$:\footnote{Note we do \emph{not} assume $B \in \mathfrak{N}$.}
\[
A \in \mathfrak{N} \cap {}^\perp B \ \implies \left( \mathfrak{N} \cap A \in {}^\perp B \text{ and } \frac{A}{\mathfrak{N} \cap A} \in {}^\perp B \right).
\]
\end{lemma}

\subsection{Kaplansky's Theorem}\label{sec_Kaplansky}

In this section we show how (a minor variant of) Theorem \ref{thm_DeconChar} yields the following classic theorem of Kaplansky:

\begin{theorem}[Kaplansky~\cite{MR0100017}]\label{thm_Kaplansky}
Every projective module is a direct sum of countably-generated projective modules.
\end{theorem}
\begin{proof}
First, notice that if $\langle P_i \ : \ i \le \zeta \rangle$ is a filtration with each $P_{i+1}/P_i$ projective, then each inclusion $P_i \to P_{i+1}$ splits; i.e., 
\[
P_{i+1} \simeq P_i \oplus \frac{P_{i+1}}{P_i}.
\]
So $P_\zeta$ is isomorphic to the direct sum
\[
\bigoplus_{i < \zeta} P_{i+1}/P_i.
\]
So it suffices to show that the class of projective modules is $\mathcal{P}_{\aleph_1}$-filtered, where $\mathcal{P}_{\aleph_1}$ is the class of countably-generated (equivalently, countably-presented) projective modules.

By a minor variant of the proof of the \ref{item_MainChar} $\implies$ \ref{item_KappaDecon} direction of Theorem \ref{thm_DeconChar} (with $\kappa = \omega_1$), it suffices to show that whenever $P$ is projective and
\[
P \in \mathfrak{N} \prec (H_\theta,\in,R)
\]
(but where $R$ is not necessarily a subset of $\mathfrak{N}$), then both $\langle \mathfrak{N} \cap P \rangle$ and $\frac{P}{\langle \mathfrak{N} \cap P\rangle}$ are projective.  Lemma 3.3 of \cite{Cox_MaxDecon} gives a direct proof of this using the Dual Basis characterization of projectivity, but we can also prove it using Lemma \ref{lem_FreeTrace}.  By elementarity of $\mathfrak{N}$, there is some free $F \in \mathfrak{N}$ such that $P$ is a direct summand of $F$; and again by elementarity we can take the complement to be in $\mathfrak{N}$.  So there is some module $X \in \mathfrak{N}$ such that $F = P \oplus X$.  Using elementarity of $\mathfrak{N}$ (and the fact that $\{ P, F, X \} \subset \mathfrak{N}$), it is routine to verify that
\begin{equation}
\langle \mathfrak{N} \cap F \rangle \simeq \langle \mathfrak{N} \cap P \rangle \oplus \langle \mathfrak{N} \cap X \rangle.
\end{equation}
Since $\langle \mathfrak{N} \cap F \rangle$ is free by Lemma \ref{lem_FreeTrace}, this shows that $\langle \mathfrak{N} \cap P \rangle$ is a direct summand of a free module, and is hence projective.  And
\[
\frac{F}{\langle \mathfrak{N} \cap F \rangle} \simeq \frac{P \oplus X}{\langle \mathfrak{N} \cap P \rangle \oplus \langle \mathfrak{N} \cap X \rangle} \simeq \frac{P}{\langle \mathfrak{N} \cap P \rangle} \oplus \frac{X}{\langle \mathfrak{N} \cap X \rangle } 
\] 
By Lemma \ref{lem_FreeTrace}, $\frac{F}{\langle \mathfrak{N} \cap F \rangle}$ is free.  Hence, $\frac{P}{\langle \mathfrak{N} \cap P \rangle}$ is a direct summand of a free module, and is hence projective.
\end{proof}

\subsection{Preservation of deconstructibility: set-sized intersections}\label{sec_SetSizedIntersect}

Recall the following theorem from the introduction:
\begin{theorem}
[\v{S}t'ov\'{\i}\v{c}ek~\cite{MR3010854}]
If $I$ is a set and $\mathcal{C}_i$ is a deconstructible class of $R$-modules for each $i \in I$, then $\bigcap_{i \in I} \mathcal{C}_i$ is deconstructible.
\end{theorem}
We provide a few details that were missing in the sketch in the introduction.  

\begin{proof}
Closure of each $\mathcal{C}_i$ under transfinite extensions implies that $\mathcal{C}:=\bigcap_{i \in I} \mathcal{C}_i$ is closed under transfinite extensions.  It remains to show that $\mathcal{C}$ is \emph{weakly} deconstructible.

For each $i \in I$ let $\kappa_i$ witness deconstructibility of $\mathcal{C}_i$, and fix a regular $\kappa$ with 
\[
\kappa > |I| + |R| + \text{sup}_{i \in I} \kappa_i.
\]
Then each $\mathcal{C}_i$ is also $\kappa$-deconstructible (since $\kappa \ge \kappa_i$).  By the \ref{item_KappaDecon} $\implies$ \ref{item_MainChar} direction of Theorem \ref{thm_DeconChar}, for all sufficiently large regular $\theta > \kappa$ and each $i \in I$, if 
\[
\mathfrak{N} \prec \mathfrak{A}^\theta_i:=\left( H_\theta,\in, R, \mathcal{C}_i \cap H_\theta \right)
\]
and $\mathfrak{N} \cap \kappa$ is transitive, then 
\[
C \in \mathfrak{N} \cap \mathcal{C}_i \ \implies \ \left( \text{ both } \mathfrak{N} \cap C \text{ and } \frac{C}{\mathfrak{N} \cap C} \text{ are in } \mathcal{C}_i \right).
\]

By the \ref{item_non_diag} $\implies$ \ref{item_KappaDecon} direction of Theorem \ref{thm_DeconChar}, to show that $\mathcal{C}$ is $\kappa$-deconstructible, it suffices to prove that for any fixed $C \in \mathcal{C}$, there is a structure in a countable signature (possibly depending on the module $C$) such that elementary submodels of it have the required reflection properties with respect to the module $C$.  Fix such a $C$, and let $\theta$ be a regular cardinal with $C \in H_\theta$ (and sufficiently large that $\mathfrak{A}_i^\theta$ is defined for each $i \in I$).  Set $\Omega:= \left( 2^{|H_\theta|} \right)^+$ and 
\[
\mathfrak{B}:=\left( H_{\Omega}, \in, R, I, \kappa, \langle \mathfrak{A}_i^\theta \ : \ i \in I \rangle, C \right).
\]
Consider any $\mathfrak{N} \prec \mathfrak{B}$ such that $\mathfrak{N} \cap \kappa$ is transitive.  Then $I$ is both an element and subset of $\mathfrak{N}$, and $\mathfrak{A}_i^\theta \in \mathfrak{N}$ for all $i \in I$.  It follows that
\[
\forall i \in I \ \ \mathfrak{N} \cap H_\theta \prec \mathfrak{A}^\theta_i,
\]
and hence, since $C \in \mathfrak{N} \cap H_\theta$, we have
\[
\forall i \in I \ \mathfrak{N} \cap C \in \mathcal{C}_i \text{ and } \frac{C}{\mathfrak{N} \cap C} \in \mathcal{C}_i.
\]
I.e., both $\mathfrak{N} \cap C$ and $\frac{C}{\mathfrak{N} \cap C}$ are in $\mathcal{C}$.
\end{proof}

\subsection{Preservation of deconstructibility: direct summands}

\begin{theorem}[\v{S}\v{t}ov\'{\i}\v{c}ek-Trlifaj~\cite{MR2476814}]\label{thm_DirectSummand}
If $\mathcal{C}$ is a decontructible class of modules, then so is the class
\[
\mathcal{C}_\oplus:= \{ M \ : \ \exists C \in \mathcal{C} \ \ M \text{ is a direct summand of } C  \}.
\]
\end{theorem}
\begin{proof}
Suppose $\mathcal{C}$ is $\kappa$-deconstructible.  Fix any $M \in \mathcal{C}_{\oplus}$; by clause \ref{item_non_diag} of Theorem \ref{thm_DeconChar}, it suffices to find a first order structure (possibly depending on $M$) whose elementary submodels reflect in the desired way.  

By definition of $\mathcal{C}_{\oplus}$, there is some $C \in \mathcal{C}$ such that $M \oplus X = C$ for some $X$.  Fix a regular $\theta$ such that $M,X,C \in H_\theta$,\footnote{For classes $\mathcal{C}$ that are $\Sigma_1$ definable over $(V,\in)$, any $H_\theta$ with $M \in H_\theta$ will automatically have such an $X$ and $C$.  But here we are \textbf{not} assuming $\mathcal{C}$ is $\Sigma_1$ definable, hence the need to explicitly choose $\theta$ large enough to contain such a $C$ and $X$, and use clause \ref{item_non_diag} of Theorem \ref{thm_DeconChar}.} and set
\[
\mathfrak{A}:= (H_\theta,\in,R,\kappa, X,C,M, \mathcal{C} \cap H_\theta).
\]
We claim that $\mathfrak{A}$ satisfies clause \ref{item_non_diag} of Theorem \ref{thm_DeconChar}, for the class $\mathcal{C}_{\oplus}$.  Consider any $\mathfrak{N} \prec \mathfrak{A}$ such that $\mathfrak{N} \cap \kappa$ is transitive.  By equivance of \ref{item_KappaDecon} with \ref{item_MainChar} for the class $\mathcal{C}$, both $\mathfrak{N} \cap C$ and $\frac{C}{\mathfrak{N} \cap C}$ are in $\mathcal{C}$.  And, similarly as in our proof of Kaplansky's Theorem (Theorem \ref{thm_Kaplansky}), 
\[
\mathfrak{N} \cap C = (\mathfrak{N} \cap M) \oplus (\mathfrak{N}\cap X)
\]
and
\[
\frac{C}{\mathfrak{N} \cap C} = \frac{M \oplus X}{\mathfrak{N} \cap (M \oplus X)} = \frac{M \oplus X}{(\mathfrak{N} \cap M) \oplus (\mathfrak{N} \cap X)} \simeq \frac{M}{\mathfrak{N} \cap M} \oplus \ \frac{X}{\mathfrak{N} \cap X}.
\]
In particular, $\mathfrak{N} \cap M$ is a direct summand of $\mathfrak{N} \cap C$, and $\frac{M}{\mathfrak{N} \cap M}$ is a direct summand of $\frac{C}{\mathfrak{N} \cap C}$.  Hence, both $\mathfrak{N} \cap M$ and $\frac{M}{\mathfrak{N} \cap M}$ are in $\mathcal{C}_\oplus$.    

We omit the proof that $\mathcal{C}_\oplus$ is closed under transfinite extensions;  it follows from the fact that $\mathcal{C}$ is closed under transfinite extensions, and has nothing to do with deconstructibility or elementary submodels. 
\end{proof}

\section{Traces of elementary submodels on more complicated objects}\label{sec_TraceComplicated}

So far, we have only discussed traces of set-theoretic elementary submodels on relatively simple objects like modules and homomorphisms of modules.  But many of the applications  will require taking traces on more complicated objects, such as complexes.  A \textbf{complex} (of $R$-modules) is a sequence
\[
\xymatrix{
\dots \ar[r] & M_{n-1} \ar[r]^-{d_{n-1}} & M_n \ar[r]^-{d_{n}} & M_{n+1} \ar[r] & \dots
}
\]
of homomorphisms indexed by $\mathbb{Z}$, with the property that $d_{n+1} \circ d_n = 0$---i.e., $\text{im} \ d_{n-1} \subseteq \text{ker} \ d_n$---for all $n \in \mathbb{Z}$.
We will typically use $M_\bullet$ to denote a complex, where the map names are typically omitted.  $\boldsymbol{Ch(R)}$ denotes the category of complexes of $R$-modules, where a morphism from a complex $A_\bullet$ to a complex $B_\bullet$ is a \textbf{chain map}, which is a sequence $f_n: A_n \to B_n$ (for $n \in \mathbb{Z}$) of $R$-module homomorphisms such that the following diagram commutes in each square:
\[
\xymatrix{
  \cdots \ar[r] & A_{n+1} \ar[d]^{f_{n+1}} \ar[r]^{d_{n+1}} & A_n \ar[d]^{f_n} \ar[r]^{d_{n}} & A_{n-1} \ar[d]^{f_{n-1}} \ar[r] & \cdots \\
  \cdots \ar[r] & B_{n+1} \ar[r]_{e_{n+1}} & B_n \ar[r]_{e_{n}} & B_{n-1} \ar[r] & \cdots 
}
\]

A \textbf{short exact sequence of complexes} is a sequence of chain maps
\begin{equation}\label{eq_SES_of_complexes}
\xymatrix{
0_\bullet \ar[r] & A_\bullet \ar[r]^-e & B_\bullet \ar[r]^-\pi &  C_\bullet \ar[r] & 0_\bullet
}
\end{equation}
such that the image of each chain map equals the kernel of the next, where the kernel of a chain map $f: X_\bullet \to Y_\bullet$ is the complex of kernels of the $f_n$'s (which can easily be shown to form a complex).  If $e: A_\bullet \to B_\bullet$ is an injective chain map, then there is a quotient object $B_\bullet / e(A_\bullet)$ yielding a short exact sequence of complexes:\footnote{Where the $n$-th module in the quotient is simply $B_n / e_n(A_n)$, each $\pi_n$ is the canonical factor map from $B_n$ onto $B_n / e_n(A_n)$, and the connecting map from $B_n / e_n(A_n)$ to $B_{n+1} / e_{n+1}(A_{n+1})$ is defined as
\[
b + e_n(A_n) \mapsto f_n(b) + e_{n+1}(A_{n+1})
\]
where $f_n: B_n \to B_{n+1}$ is the connecting map in the complex $B_\bullet$.}
\[
\xymatrix{
0_\bullet \ar[r] &  A_\bullet  \ar[r]^-e & B_\bullet \ar[r]^-\pi  & B_\bullet / e(A_\bullet) \ar[r] & 0_\bullet 
}
\]

A complex $M_\bullet$ is  \textbf{acyclic} if it is exact at every coordinate, i.e., $\text{im} \ d_{n-1} = \text{ker} \ d_n$ for all $n$ (recall that being a complex only requires the $\subseteq$ direction).  The individual complexes in a short exact sequence of complexes, as in \eqref{eq_SES_of_complexes}, may or may not be acyclic complexes.  The following basic tool from homological algebra (proved via straightforward diagram chases) is frequently used to determine when the complexes are acyclic.

\begin{lemma}[3-by-3 lemma for complexes, Weibel~\cite{MR1269324}]\label{lem_3_by_3_lemma}
If 
\[
\xymatrix{
0_\bullet \ar[r] & A_\bullet \ar[r] & B_\bullet \ar[r] & C_\bullet \ar[r] & 0_\bullet
}
\]
is a short exact sequence of complexes, and two out of the members of the set $\{ A_\bullet, B_\bullet, C_\bullet \}$ are acyclic, then so is the third.
\end{lemma}

If $A_\bullet$ is acyclic, subcomplexes of $A_\bullet$ are not necessarily acyclic.  But part \eqref{item_AcyclicReflects} of the following lemma implies that ``almost all" subcomplexes of $A_\bullet$ are acyclic:

\begin{lemma}\label{lem_RestrictComplex}
Suppose $R$ is a ring, $\theta$ is a regular uncountable cardinal, and $\mathfrak{N} \prec H_\theta$ is such that $R$ is both an element and subset of $\mathfrak{N}$.  If $A_\bullet = \big( 
\begin{tikzcd}
A_n \arrow[r, "d_n"] &  A_{n+1} 
\end{tikzcd}
\big)_{n \in \mathbb{Z}}$ is a complex of $R$-modules and $A_\bullet \in \mathfrak{N}$, then:
\begin{enumerate}
 \item The following sequence---which we denote by $A_\bullet \restriction \mathfrak{N}$---is a complex of $R$-modules:
\[
\begin{tikzcd}
\dots \arrow[r] & \mathfrak{N} \cap A_n \arrow[r, "d_n \restriction (\mathfrak{N} \cap A_n)"] &[3em] \mathfrak{N} \cap A_{n+1} \arrow[r] & \dots
\end{tikzcd} 
\]
 \item\label{item_AcyclicReflects}  $A_\bullet$ is acyclic if and only if $A_\bullet \restriction \mathfrak{N}$ is acyclic.

 \item\label{item_InducedSES_acyclic} 
Let $A_\bullet / \mathfrak{N}$ denote the quotient of the inclusion $A_\bullet \restriction \mathfrak{N} \to A_\bullet$, and consider the following short exact sequence of complexes:
 \[
 \xymatrix{
 0_\bullet \ar[r] & A_\bullet \restriction \mathfrak{N} \ar[r] & A_\bullet \ar[r] & A_\bullet / \mathfrak{N} \ar[r] & 0_\bullet
 }
 \]
If $A_\bullet$ is acyclic, then all terms in this short exact sequence are acyclic.
\end{enumerate}
\end{lemma}
\begin{proof}
Since $A_\bullet \in \mathfrak{N}$, $d_n \in \mathfrak{N}$ for all $n \in \mathbb{Z}$ by elementarity of $\mathfrak{N}$.  Hence, $d_n \restriction \mathfrak{N} \cap A_n$ maps into $\mathfrak{N} \cap A_{n+1}$.  Clearly this is a complex (consecutive maps compose to 0) because $A_\bullet$ has this property.  

If $A_\bullet$ happens to be acyclic, this is also inherited by $A_\bullet \restriction \mathfrak{N}$, since if $x \in \mathfrak{N} \cap \text{ker} d_n$, then by exactness of $A_\bullet$ at $n$ and elementarity of $\mathfrak{N}$, there is some $z \in \mathfrak{N} \cap A_{n-1}$ such that $d_{n-1}(z)=x$ (note this is a key use of elementarity; subcomplexes of acyclic complexes are \emph{not} always acyclic!).  Similarly, if $A_\bullet$ is not acyclic, then there is some $n \in \mathbb{Z}$ such that $\text{ker}  \ d_n \ne \text{im} \  d_{n-1}$; since $\mathbb{Z} \subset \mathfrak{N}$ and $A_\bullet \in \mathfrak{N}$, the witness to the non-equality can be taken in $\mathfrak{N}$.\footnote{Alternatively, one can observe that ``$A_\bullet$ is acyclic in $\text{Ch}(R)$" can be expressed in a $\Sigma_0$ manner, and use that $A_\bullet \restriction \mathfrak{N}$ is isomorphic, in $\text{Ch}(R)$, to its image under the transitive collapsing map of $\mathfrak{N}$.}  This proves part \eqref{item_AcyclicReflects}.  Part \eqref{item_InducedSES_acyclic} follows from part \eqref{item_AcyclicReflects} and the 3-by-3 lemma.
\end{proof}

The notion of deconstructibility makes sense for classes of complexes too, with the obvious modifications (and even for Grothendieck categories and beyond, e.g., \v{S}t'ov\'{\i}\v{c}ek~\cite{MR3010854}).  When appropriately translated, Theorem \ref{thm_DeconChar} generalizes to abelian categories that are locally finitely presentable (and a bit beyond).  But for concreteness, we stick to a version for complexes:

\begin{theorem}[version of Theorem \ref{thm_DeconChar} for classes in $\text{Ch}(R)$]\label{thm_DeconCharComplexes}
Suppose $R$ is a ring and $\kappa$ is a regular uncountable cardinal with $\kappa > |R|$.  Suppose $\mathcal{C}$ is a class of complexes in $\text{Ch}(R)$ that is closed under filtrations.  Then all equivalences of Theorem \ref{thm_DeconChar} hold, with each ``$\mathfrak{N} \cap C$" replaced by ``$C_\bullet \restriction \mathfrak{N}$" (for $C_\bullet \in \mathfrak{N} \cap \mathcal{C}$).
\end{theorem}

One could weaken the assumption $\kappa > |R|$ to ``$R$ is $\kappa$-Noetherian", by making use of Theorem \ref{thm_CharNoeth}; see \cite{Cox_MaxDecon} for details.

\section{More applications, old and new}\label{sec_MoreApps}

\subsection{Resolutions by deconstructible classes}\label{sec_Resolutions}

Homological algebra uses \emph{projective resolutions} to approximate non-projective modules.  More generally, given an $R$-module $M$ and a class $\mathcal{F}$ of $R$-modules, an \textbf{$\boldsymbol{\mathcal{F}}$-resolution of $\boldsymbol{M}$} is a long exact sequence of the form
\[
\begin{tikzcd}
\dots \arrow[r] & F_n \arrow[r,"\pi_n"] & F_{n-1} \arrow[r]  & \dots \arrow[r] & F_1 \arrow[r,"\pi_1"] & F_0 \arrow[r,"\pi_0"] & M \arrow[r] & 0
\end{tikzcd}
\]
where each $F_i$ is either 0 or an element of $\mathcal{F}$.  A module $M$ has \textbf{$\boldsymbol{\mathcal{F}}$-dimension at most $\boldsymbol{n}$} if it has an $\mathcal{F}$-resolution such that $F_i = 0$ for all $i > n$.  $M$ has \textbf{finite $\mathcal{F}$-dimension} if it has $\mathcal{F}$-dimension at most $n$, for some natural number $n$.

Kaplansky's Theorem \ref{thm_Kaplansky} implies that the class of modules with projective dimension 0---i.e., the projective modules---is deconstructible.  Theorem \ref{thm_FlatCover} (proof of the Flat Cover Conjecture) says the same is true for the class of flat modules.  These facts were generalized by Aldrich et al.~\cite{MR1848954} to the classes of modules of projective or flat dimension at most $n$, and this was further strengthened by Sl\'{a}vik-Trlifaj~\cite{MR3161764}, restated as Theorem \ref{thm_SlavikTrlifaj} below.  It says that if $\langle \mathcal{C}_n \ : \ n < \omega \rangle$ is a sequence of $\kappa$-deconstructible classes of modules, then the class of modules possessing some ``$\vec{\mathcal{C}}$-resolution"---i.e., those modules $M$ for which there exists an exact sequence
\[
\begin{tikzcd}
\dots C_2 \arrow[r] & C_1 \arrow[r] & C_0 \arrow[r] & M \arrow[r] & 0
\end{tikzcd}
\]
with each $C_n \in \mathcal{C}_n$---is $\kappa$-deconstructible.\footnote{This implies the results of Aldrich et al.~\cite{MR1848954} by taking $\mathcal{C}_i:=$ either the class of projectives or the class of flats for each $i \le n$, and $\mathcal{C}_i = \{ 0 \}$ for $i > n$.}  More precisely:
\begin{theorem}[Sl\'{a}vik-Trlifaj~\cite{MR3161764}]\label{thm_SlavikTrlifaj}
Suppose $R$ is a $\kappa$-Noetherian ring and $\langle S_n \ : \ n < \omega \rangle$ is a sequence of sets of $\kappa$-presented $R$-modules.  Set $\mathcal{C}_n:= \text{Filt}(S_n)$ for each $n \in \omega$, and let $\mathcal{D}$ denote the class of modules that possess a $\vec{\mathcal{C}}$-resolution.  Then $\mathcal{D}$ is $\kappa$-deconstructible.
\end{theorem}

We use Theorem \ref{thm_DeconChar} to provide a new, short proof of Theorem \ref{thm_SlavikTrlifaj} (but under the simplifying assumption $|R|<\kappa$; the full version of Theorem \ref{thm_SlavikTrlifaj} can be proved using the version of Theorem \ref{thm_DeconChar} from Cox~\cite{Cox_MaxDecon}).

\begin{proof}
First we check weak $\kappa$-deconstructibility.  Suppose $D \in \mathcal{D}$, as witnessed by some $\vec{\mathcal{C}}$-resolution
\[
\begin{tikzcd}
\dots \arrow[r] C_2 \arrow[r] & C_1 \arrow[r] & C_0 \arrow[r] & D \arrow[r] & 0 
\end{tikzcd}
\]
This can be viewed as a (truncated) acyclic complex, which we will denote by $X_\bullet$.

Fix a regular $\theta$ with all these parameters in $H_\theta$, and consider any 
\[
\mathfrak{N} \prec \mathfrak{A}:=(H_\theta,\in,R,\kappa, \langle S_n \ : \ n < \omega \rangle)
\]
such that $X_\bullet \in \mathfrak{N}$ and $\mathfrak{N} \cap \kappa$ is transitive.  By the \ref{item_non_diag} $\implies$ \ref{item_KappaDecon} direction of Theorem \ref{thm_DeconChar}, it suffices to prove that both $\mathfrak{N} \cap D$ and $\frac{D}{\mathfrak{N} \cap D}$ have $\vec{\mathcal{C}}$-resolutions.  By Lemma \ref{lem_RestrictComplex}, since $X_\bullet$ is acyclic, then
\[
\begin{tikzcd}
0_\bullet \arrow[r] & X_\bullet \restriction \mathfrak{N} \arrow[r] & X_\bullet \arrow[r] & X_\bullet / \mathfrak{N} \arrow[r] & 0_\bullet
\end{tikzcd}
\]
is a short exact sequence of acyclic complexes, i.e., the following has all exact rows and columns:
\begin{equation}\label{eq_GridOfResolutions}
\begin{tikzcd}
0_\bullet: \arrow[d] & \dots & 0 \arrow[d] \arrow[r] & 0 \arrow[d] \arrow[r] & 0 \arrow[r] \arrow[d] & 0 \arrow[d] \\
X_\bullet \restriction \mathfrak{N} : \arrow[d] & \dots &\mathfrak{N} \cap  C_1 \arrow[d] \arrow[r] & \mathfrak{N} \cap  C_0 \arrow[d] \arrow[r] & \mathfrak{N} \cap  D \arrow[d] \arrow[r] & 0 \arrow[d] \\
X_\bullet: \arrow[d] & \dots & C_1 \arrow[r] \arrow[d] & C_0 \arrow[r] \arrow[d] & D \arrow[r] \arrow[d] & 0 \arrow[d]\\
X_\bullet/\mathfrak{N}: \arrow[d] & \dots & \frac{C_1}{\mathfrak{N} \cap C_1} \arrow[r] \arrow[d] & \frac{C_0}{\mathfrak{N} \cap C_0} \arrow[r] \arrow[d] & \frac{D}{\mathfrak{N} \cap D} \arrow[r] \arrow[d] & 0 \arrow[d] \\ 
0_\bullet:  & \dots & 0  \arrow[r] & 0  \arrow[r] & 0 \arrow[r]  & 0 
\end{tikzcd}
\end{equation}
Since $C_n \in \mathcal{C}_n \cap \mathfrak{N}$ for each $n \in \omega$ and $\mathcal{C}_n$ was $\kappa$-deconstructible, the \ref{item_KappaDecon} $\implies$ \ref{item_MainChar} direction of Theorem \ref{thm_DeconChar} tells us that
\[
\forall n \in \omega \ \ \mathfrak{N} \cap C_n \in \mathcal{C}_n \text{ and } \frac{C_n}{\mathfrak{N} \cap C_n} \in \mathcal{C}_n. 
\] 
So the row of \eqref{eq_GridOfResolutions} corresponding to $X_\bullet \restriction \mathfrak{N}$ is a $\vec{\mathcal{C}}$-resolution of $\mathfrak{N} \cap D$, and the row corresponding to $X_\bullet / \mathfrak{N}$ is a $\vec{\mathcal{C}}$-resolution of $\frac{D}{\mathfrak{N} \cap D}$.

This completes the proof that $\mathcal{D}$ is weakly $\kappa$-deconstructible. We omit the proof that $\mathcal{D}$ is filtration-closed; this is a purely algebraic proof, using the fact that each $\mathcal{C}_n$ is a precovering class (by Theorem \ref{thm_DeconPrecover}).

\end{proof}

\subsection{Almost everywhere closure under quotients}\label{sec_AE_closure}

In \cite{Cox_MaxDecon} we isolated the following property, after observing that classes of complexes associated with \emph{Gorenstein Homological Algebra} typically possess it.  It is a very weak version of closure of a class under quotients: 
\begin{definition}[\cite{Cox_MaxDecon}]
Fix a ring $R$, a class $\mathcal{C} \subseteq R$-Mod, and a regular uncountable cardinal $\kappa$.  We say $\mathcal{C}$ is \textbf{$\boldsymbol{\kappa}$-a.e.\ closed under quotients} if whenever
\[
\mathfrak{N} \prec \mathfrak{N}' \prec (H_\theta,\in,R,\mathcal{C} \cap H_\theta)
\]
and both $\mathfrak{N} \cap \kappa$ and $\mathfrak{N}' \cap \kappa$ are transitive, then for every $C \in \mathfrak{N}$ the following implication holds:
\begin{equation*}
\left( C \in \mathcal{C} \text{ and }  \mathfrak{N} \cap C \in \mathcal{C} \text{ and } \mathfrak{N}' \cap C \in \mathcal{C} \right) \  \implies \ \frac{\mathfrak{N}' \cap C}{\mathfrak{N} \cap C} \in \mathcal{C}.
\end{equation*}

We also define this notion for classes $\mathcal{C} \subseteq \text{Ch}(R)$ of complexes, where for $C_\bullet \in \mathfrak{N}$ the implication above is replaced by:
\begin{equation*}
\left( C_\bullet \in \mathcal{C} \text{ and }  C_\bullet \restriction \mathfrak{N} \in \mathcal{C} \text{ and } C_\bullet \restriction \mathfrak{N}'  \in \mathcal{C} \right) \  \implies \ \frac{C_\bullet \restriction \mathfrak{N}'}{C_\bullet \restriction \mathfrak{N}} \in \mathcal{C}.
\end{equation*}

We say $\mathcal{C}$ is \textbf{eventually a.e.\ closed under quotients} if it is $\kappa$-a.e.\ closed under quotients for all sufficiently large regular $\kappa$.
\end{definition}

\begin{lemma}\label{lem_DeconImpliesClosure}
If $\mathcal{C}$ is a deconstructible class of modules or complexes, then it is eventually a.e.\ closed under quotients.
\end{lemma}
\begin{proof}
Suppose $\mathcal{C}$ is a $\kappa_0$-deconstructible class of modules; then it is $\kappa$-deconstructible for all regular $\kappa \ge \kappa_0$.  Then for any $C \in \mathfrak{N} \prec \mathfrak{N}'$ as in the statement of the lemma (with $\mathfrak{N} \cap \kappa$ and $\mathfrak{N}' \cap \kappa$ both transitive), the  \ref{item_KappaDecon} $\implies$ \ref{item_DiagVersion2models} direction of Theorem \ref{thm_DeconChar} immediately yields that $\frac{\mathfrak{N}' \cap C}{\mathfrak{N} \cap C} \in \mathcal{C}$. 

If $\mathcal{C}$ is a class of complexes, the argument is the same, except one uses Theorem \ref{thm_DeconCharComplexes} instead of Theorem \ref{thm_DeconChar}.
\end{proof}

If $\mathcal{C}$ is eventually a.e.\ closed under quotients, and you're trying to prove that it is deconstructible, then the task of trying to verify clause \ref{item_MainChar} (or clause \ref{item_DiagVersion2models}) in the characterization of deconstructibility (Theorem \ref{thm_DeconChar}) becomes easier, because one only has to verify the ``$\mathfrak{N} \cap C \in \mathcal{C}$" part of that clause; the ``$\frac{C}{\mathfrak{N} \cap C} \in \mathcal{C}$" part is then taken care of by eventual a.e.\ closure of $\mathcal{C}$ under quotients, by taking $\mathfrak{N}':= H_\theta$.  

But in fact, eventual a.e.\ closure under quotients makes the task even easier for us: it allows us to turn a stationary set of reflections to a closed unbounded one (assuming filtration-closure of the class), and then to extend further to larger elementary submodels.

\begin{theorem}[essentially Theorem 6.3 of \cite{Cox_MaxDecon}]\label{thm_TurnStatToClub}
Suppose $\mathcal{C}$ is a class of $R$-modules, $\kappa$ is a regular uncountable cardinal larger than $|R|$, and:
\begin{enumerate}[label=(\alph*)]
 \item\label{item_FiltClosure} $\mathcal{C}$ is filtration-closed;
 \item\label{item_Assump_AEclosed} $\mathcal{C}$ is $\kappa$-a.e.\ closed under quotients; and
 \item\label{item_StatReflect} For every $C \in \mathcal{C}$, the collection
\[
\mathcal{C} \cap \wp_\kappa(C)
\]
is stationary in $\wp_\kappa(C)$.    

\end{enumerate}
Then $\mathcal{C}$ is $\kappa$-deconstructible.

\textbf{Note:}  if $\mathcal{C}$ is a class of complexes instead of modules, the same theorem holds, with clause \ref{item_StatReflect} replaced by the assumption:  For every $C_\bullet \in \mathcal{C}$, the collection\footnote{If $C_\bullet$ is a complex, the literal collection $\wp_\kappa(C_\bullet)$ would not be very useful, since it would not even be a collection of complexes.}
\[
\big\{ \mathfrak{N} \in \wp_\kappa(H_\theta) \ : \ C_\bullet \in \mathfrak{N} \text{ and } C_\bullet \restriction \mathfrak{N} \in \mathcal{C}  \big\}
\]
is stationary in $\wp_\kappa(H_\theta)$ for all sufficiently large $\theta$.
\end{theorem}
\begin{proof}
(Sketch)  Assumptions \ref{item_FiltClosure} and \ref{item_Assump_AEclosed} ensure that \textbf{whenever}
\[
\vec{\mathfrak{N}} = \langle \mathfrak{N}_\alpha \ : \ \alpha < \zeta \rangle
\]
is a $\subseteq$-continuous (but not necessarily $\in$-increasing) sequence of elementary submodels of $(H_\theta,\in,R,\kappa, \mathcal{C} \cap H_\theta)$ with each $\mathfrak{N}_\alpha \cap \kappa$ transitive, \textbf{then} the following implication holds for all $C \in \mathfrak{N}_0 \cap \mathcal{C}$, where $\mathfrak{N}:= \bigcup_{\alpha < \zeta} \mathfrak{N}_\alpha$:
\begin{equation}\label{eq_KeyImplication}
\left( \mathfrak{N}_0 \cap C \in \mathcal{C} \text{ and }  \forall \alpha < \zeta \  \frac{\mathfrak{N}_{\alpha+1} \cap C}{\mathfrak{N}_\alpha \cap C} \in \mathcal{C} \right) \implies \ \left( \mathfrak{N} \cap C \in \mathcal{C} \text{ and } \frac{C}{\mathfrak{N} \cap C } \in \mathcal{C} \right). 
\end{equation}

To see why this implication holds, assume its hypothesis; then  $\mathfrak{N} \cap C \in \mathcal{C}$ because of assumption \ref{item_FiltClosure}.  Now set $\mathfrak{N}':= H_\theta$; it follows from assumption \ref{item_Assump_AEclosed} that $\frac{\mathfrak{N}' \cap C}{\mathfrak{N} \cap C} = \frac{C}{\mathfrak{N} \cap C}$ is an element of $\mathcal{C}$.

Then one can prove, by induction on $|\mathfrak{N}|$ and using \eqref{eq_KeyImplication}, that clause \ref{item_MainChar} of Theorem \ref{thm_DeconChar} holds.  Assumption \ref{item_StatReflect} is only used to deal with the base case $|\mathfrak{N}|<\kappa$. 
\end{proof}

\subsection{Further extensions of the Flat Cover Conjecture}\label{sec_Further_EklofTrlifaj}

Recall that the Eklof-Trlifaj Theorem \ref{thm_EklofTrlifaj} was a generalization of the Flat Cover Conjecture (Theorem).  We consider strengthenings of Theorem \ref{thm_EklofTrlifaj}, the first due to Cort{\'e}s-Izurdiaga and \v{S}aroch (Theorem \ref{thm_Saroch_Cortes}), and a new one due to the author (Theorem \ref{thm_CoxGeneralizeET}).  In all three theorems, one assumes $\lambda$ is an infinite regular cardinal, $\mathcal{B}$ is a class of $\lambda$-pure-injective modules, and one wants to prove (possibly with an additional assumption) that ${}^\perp \mathcal{B}$ is deconstructible.  Since ${}^\perp \mathcal{B}$ is filtration-closed by Eklof's Lemma, the problem reduces (by Theorem \ref{thm_DeconChar}) to showing that for some sufficiently large $\kappa$, the implication
\begin{equation}\label{eq_TheMainPerpImp}
A \in {}^\perp \mathcal{B} \cap \mathfrak{N} \ \implies \left( \mathfrak{N} \cap A \in {}^\perp \mathcal{B} \text{ and } \frac{A}{\mathfrak{N} \cap A} \in {}^\perp \mathcal{B} \right)
\end{equation}
holds for almost every $\mathfrak{N} \prec (H_\theta,\in,\dots)$ whose intersection with $\kappa$ is transitive.

In our proof of the Eklof-Trlifaj Theorem \ref{thm_EklofTrlifaj}, where $\lambda = \omega$, the implication \eqref{eq_TheMainPerpImp} held automatically for \emph{any} $\mathfrak{N} \prec (H_\theta,\in)$, as long as $R \in \mathfrak{N}$ and $\mathfrak{N} \cap |R|^+$ was transitive.  Lemma \ref{lem_ElemPureReflect} was the key there, and that lemma was ultimately due to the fact that \emph{every} elementary submodel of $(H_\theta,\in)$ is closed under finite subsets (Fact \ref{fact_BasicFactsElemSub} part \eqref{item_FiniteSubset}).  I.e., $\omega$-purity reflects to \emph{every} $\mathfrak{N} \prec (H_\theta,\in)$, as long as $R$ is both an element and subset of $\mathfrak{N}$.

Remark \ref{rem_LambdaPureLambdaclosed} says that if $\mathfrak{N} \prec (H_\theta,\in)$ is closed under $<\lambda$ sequences, then its trace on any $A \in \mathfrak{N}$ is $\lambda$-pure in $A$.  If $\kappa^{<\lambda} = \kappa$, the collection of such $\mathfrak{N}$ is stationary in $\wp_\kappa(H_\theta)$, but it is also co-stationary if $\lambda \ge \omega_1$.\footnote{Because, for example, there are stationarily many $\mathfrak{N}$ whose intersection with $\kappa$ is $\omega$-cofinal; such $\mathfrak{N}$ are not $<\omega_1$-closed.}  In this situation, the implication \eqref{eq_TheMainPerpImp} is guaranteed to hold for $<\lambda$-closed $\mathfrak{N} \in \wp_\kappa(H_\theta)$,\footnote{Via the proof of Theorem \ref{thm_EklofTrlifaj}, replacing each instance of ``pure" with ``$\lambda$-pure" and using Remark \ref{rem_LambdaPureLambdaclosed} instead of Lemma \ref{lem_ElemPureReflect}.} which at least ensures that ${}^\perp \mathcal{B}$ is a \textbf{Kaplansky Class}.  But being a Kaplansky class is weaker, in general, than deconstructibility (and does \emph{not} abstractly imply the class is precovering).  To achieve deconstructibility, one also needs the implication \eqref{eq_TheMainPerpImp} to hold even for those $\mathfrak{N}$ that are not $<\lambda$-closed.  Theorems \ref{thm_Saroch_Cortes} and \ref{thm_CoxGeneralizeET} employ different assumptions to achieve this; the former assumes closure of ${}^\perp \mathcal{B}$ under direct limits, while the latter assumes ${}^\perp \mathcal{B}$ is eventually a.e.\ closed under quotients.

\begin{theorem}[Cort{\'e}s-Izurdiaga and \v{S}aroch \cite{cortes2023module}, Theorem 4.6 and Corollary 4.7]\label{thm_Saroch_Cortes}
Let $\lambda$ be an infinite regular cardinal.  Suppose $\mathcal{B}$ is a class of $\lambda$-pure-injective modules, and that
\begin{equation}\label{eq_CortesDagger}
{}^\perp \mathcal{B} \text{ is closed under direct limits.}
\tag{$\dagger$}
\end{equation}
Then ${}^\perp \mathcal{B}$ is deconstructible.
\end{theorem}

One could re-prove Theorem \ref{thm_Saroch_Cortes} in a way similar to our proof of the Eklof-Trlifaj Theorem \ref{thm_EklofTrlifaj}.  The argument from the proof of Theorem \ref{thm_EklofTrlifaj}, together with Remark \ref{rem_LambdaPureLambdaclosed}, would imply that (under the assumptions of Theorem \ref{thm_Saroch_Cortes}) 
\[
A \in \mathfrak{N} \cap {}^\perp \mathcal{B} \ \implies \ \left( \mathfrak{N} \cap A \in {}^\perp \mathcal{B} \text{ and } \frac{A}{\mathfrak{N} \cap A} \in {}^\perp \mathcal{B} \right)
\]
holds whenever $\mathfrak{N}$ is closed under sequences of length $<\lambda$.  Then one could use assumption \eqref{eq_CortesDagger} to achieve the same implication for almost every $\mathfrak{N}$ of any size, by induction on $|\mathfrak{N}|$ and using filtrations of $\mathfrak{N}$.\footnote{One just views $\mathfrak{N}$ as the union of a smooth sequence $\big\langle \mathfrak{N}_i \ : \ i < \text{cf}(|\mathfrak{N}|) \big\rangle$ of elementary submodels, each of size $<|\mathfrak{N}|$.  Then for any $A \in \mathfrak{N}$,  $\mathfrak{N} \cap A$ is the direct limit of $\big\langle \mathfrak{N}_i \cap A \ : \ i < \text{cf}(|\mathfrak{N}|) \big\rangle$, and $\frac{A}{\mathfrak{N} \cap A}$ is the direct limit of $\left\langle \frac{A}{\mathfrak{N}_i \cap A} \ : \ i < \text{cf}(|\mathfrak{N}|) \right\rangle$.  It is important here that the assumption \eqref{eq_CortesDagger} is required to hold of all directed systems, not just directed systems of monomorphisms.}

We strengthen Theorem \ref{thm_Saroch_Cortes}:  
\begin{theorem}\label{thm_CoxGeneralizeET}
Theorem \ref{thm_Saroch_Cortes} holds when the assumption \eqref{eq_CortesDagger} is replaced by:
\begin{equation}\label{eq_MyNewThm}
{}^\perp \mathcal{B} \text{ is eventually a.e. closed under quotients}. \tag{$\dagger \dagger$}
\end{equation}
\end{theorem}

\begin{remark}
We can also weaken the assumption on $\mathcal{B}$, to only require injectivity with respect to inclusions of the form
\[
\mathfrak{N} \cap A \to A
\]
where $A \in \mathfrak{N} \prec (H_\theta,\in)$ and $\mathfrak{N}$ is closed under $<\lambda$ sequences (all such inclusions are $\lambda$-pure, by Remark \ref{rem_LambdaPureLambdaclosed}).
\end{remark}

\begin{proof}
(of Theorem \ref{thm_CoxGeneralizeET}) Fix a sufficiently large regular $\kappa$ such that:
\begin{itemize}
 \item ${}^\perp \mathcal{B}$ is $\kappa$-a.e. closed under quotients; and
 \item $\alpha^{<\lambda} <\kappa$ for all $\alpha < \kappa$.
\end{itemize}
Then by Fact \ref{fact_BasicFactsElemSub}, for any $\theta > \kappa$ the set 
\[
S_{\kappa,\lambda}:=\{ \mathfrak{N} \prec (H_\theta,\in) \ : \ |\mathfrak{N}|<\kappa, \ \mathfrak{N} \cap \kappa \in \kappa, \text{ and } {}^{<\lambda} \mathfrak{N} \subset \mathfrak{N} \}
\]
is stationary in $\wp_\kappa(H_\theta)$.  By Lemma \ref{lem_LambdaClosedLeftPerp}, for every $\mathfrak{N} \in S_{\kappa,\lambda}$,
\begin{equation*}\label{eq_F_On_N}
A \in \mathfrak{N} \cap {}^\perp \mathcal{B} \ \implies \  \mathfrak{N} \cap A \in {}^\perp \mathcal{B}.
\end{equation*}
So 
\[
\{ \mathfrak{N} \cap A \ : \ A \in \mathfrak{N} \in S_{\kappa,\lambda}  \}
\]
is a stationary subset of ${}^\perp \mathcal{B} \cap \wp_\kappa(A)$.  By Theorem \ref{thm_TurnStatToClub}, ${}^\perp \mathcal{B}$ is $\kappa$-deconstructible.

\end{proof}

The assumptions of Theorem \ref{thm_CoxGeneralizeET} follow from those of Theorem \ref{thm_Saroch_Cortes}:  suppose $\mathcal{B}$ satisfies the assumptions of Theorem \ref{thm_Saroch_Cortes}.  Then, by that theorem, ${}^\perp \mathcal{B}$ is deconstructible, and hence by Lemma \ref{lem_DeconImpliesClosure}, ${}^\perp \mathcal{B}$ is eventually a.e.\ closed under quotients.  

However, under ``smallness" assumptions on the universe---in particular, if $0^\sharp$ does not exist---then the assumptions of Theorems \ref{thm_CoxGeneralizeET},  \ref{thm_Saroch_Cortes}, and \ref{thm_EklofTrlifaj} are equivalent.  This follows from a theorem of Cort{\'e}s-Izurdiaga and \v{S}aroch~\cite{cortes2023module}, which states that in the presence of enough nonreflecting stationary sets (in particular, if $0^\sharp$ does not exist), $\lambda$-pure-injectivity is equivalent to ordinary $\aleph_0$-pure injectivity, for any $\lambda$ and over any ring.  In such a model, if $\mathcal{B}$ is a class of $\lambda$-pure-injectives, then they are pure-injective, and it follows that ${}^\perp \mathcal{B}$ is closed under direct limits by a result of Angeleri, Mantese, Tonolo, and Trlifaj (cf.\ proof of Lemma 9 of \cite{MR1778163}).  And by Eklof-Trlifaj's Theorem \ref{thm_EklofTrlifaj}, ${}^\perp \mathcal{B}$ is deconstructible.

\subsection{Gorenstein Homological Algebra}\label{sec_Gorenstein}

As mentioned in the introduction, one goal of approximation theory is to do homological algebra relative to certain classes, in place of the classical ones (e.g. projective and injective modules).  Some of the most commonly-encountered classes in the recent literature are the Gorenstein Projective ($\mathcal{GP}$), Gorenstein Flat ($\mathcal{GF}$), and Gorenstein Injective ($\mathcal{GI}$) modules (Enochs-Jenda~\cite{MR1753146}, Iacob~\cite{MR3839274}); these are introduced after Definition \ref{def_AcCF} below.  Early on, the question was raised of whether analogues of Kaplansky's Theorem \ref{thm_Kaplansky} and/or the Flat Cover Theorem \ref{thm_FlatCover} hold in the Gorenstein setting.  In particular, the following question appears many times in the literature:\footnote{E.g., \cite{MR2529328}, \cite{MR3621667}, \cite{MR1363858}, \cite{MR1753146}, \cite{MR3690524},  \cite{MR3598789}, \cite{MR3459032}, \cite{MR3760311}, \cite{MR2038564}, \cite{MR3839274}, \cite{MR4115324}, \cite{MR2283103},  \cite{MR2737778}, \cite{MR3473859}, \cite{MR4166732}.}

\begin{question}\label{q_GP_GF_precover}
Is $\mathcal{GP}$ and/or $\mathcal{GF}$ always a precovering class (over every ring)?
\end{question}

For $\mathcal{GF}$ the answer is affirmative, over ZFC alone, for all rings, as shown by Yang and Liang (see Theorem \ref{thm_YangLiangDeconPrecover} below).  For $\mathcal{GP}$, the answer is affirmative in ZFC for certain kinds of rings, and for all rings if there are large cardinals (\cite{cortes2023module}, \cite{Cox_MaxDecon}), but it is still open in general.  Section \ref{sec_Gor_ZFC} focuses on what's known to be provable in ZFC alone about Question \ref{q_GP_GF_precover}, while Section \ref{sec_MaxDecon} discusses what can be done with large cardinal assumptions.

\subsubsection{ZFC results}\label{sec_Gor_ZFC}

We will use Theorems \ref{thm_DeconChar} and \ref{thm_DeconCharComplexes} to re-prove the following theorems (see below for the meaning of $\mathcal{X}$-$\mathcal{GF}$ and $\mathcal{X}$-$\mathcal{GP}$):
\begin{theorem}[Estrada-Iacob-P\'{e}rez~\cite{MR4132088}]\label{thm_EstradaEtAl}
Any class of the form $\mathcal{X}$-$\mathcal{GF}$ is a precovering class.  In particular, $\mathcal{GF}$ is always a precovering class (that $\mathcal{GF}$ is precovering---i.e., that $\mathcal{X}$-$\mathcal{GF}$ is precovering when $\mathcal{X}$ is the class of injectives---was first proved by Yang-Liang~\cite{MR3178060}).
\end{theorem}

\begin{theorem}\label{thm_GP_ZFC}
For any infinite regular $\lambda$ and any class $\mathcal{X}$ of $\lambda$-pure-injective modules, the class $\mathcal{X}$-$\mathcal{GP}$ is deconstructible. 
\end{theorem}

\noindent Theorem \ref{thm_GP_ZFC} modestly strengthens a result from Cort{\'e}s-Izurdiaga and \v{S}aroch \cite{cortes2023module}, who proved it under the additional assumption that $\mathcal{X}$ contains all the projective modules.  In particular, their result showed that $\mathcal{GP}$ is always deconstructible over left $\Sigma$-pure-injective rings (which includes, for example, rings that are simultaneously right coherent and left perfect; see \cite{cortes2023module}).

\begin{definition}\label{def_AcCF}
If $\mathcal{C}$ is a class of acyclic complexes in $\text{Ch}(R)$ and $\mathcal{F}$ is a (possibly class-indexed)\footnote{Since a single functor is a class-sized object, working in ZFC requires some care that membership in $\mathcal{F}$ is uniformly definable.  This is the case in the Gorenstein applications, since in those situations, $\mathcal{F}$ is (e.g.) the class of all functors of the form $- \otimes X$ for $X$ in a fixed class $\mathcal{X}$ of modules.} collection of functors from $\text{Ch}(R) \to \text{Ch}(\mathbb{Z})$, set
\[
\text{Ac}(\mathcal{C},\mathcal{F}):= \Big\{ C_\bullet \in \mathcal{C} \ : \ \forall F \in \mathcal{F} \ F(C_\bullet) \text{ is acyclic}  \Big\}.
\]
\end{definition}
\begin{definition}
If $\mathcal{D}$ is a class of complexes of $R$-modules,
\[
\text{Cyc}(\mathcal{D}):=\big\{ C \in R \text{-Mod} \ : \ \exists D_\bullet \in \mathcal{D} \ C = \text{ker}(D_0 \to D_1)  \big\}.
\]
\end{definition}

All classes in the literature on Gorenstein Homological algebra known to the author are of the form $\text{Cyc} \left( \text{Ac}(\mathcal{C},\mathcal{F}) \right)$ for some class $\mathcal{C}$ of complexes and some collection $\mathcal{F}$ of functors.  For example:

\begin{enumerate}[label=(\Alph*)]
 \item\label{item_Def_XGP} For a fixed class $\mathcal{X}$ of modules:  $\mathcal{X}$-$\mathcal{GP}$ is $\text{Cyc}\left( \text{Ac}(\mathcal{C}_{\text{proj}}, \text{Hom}(-,\mathcal{X})  )\right)$, where $\mathcal{C}_{\text{proj}}$ is the class of acyclic complexes with projective entries, and $\text{Hom}(-,\mathcal{X})$ is the collection of (contravariant) functors of the form 
 \[
 \text{Hom}(-,X): \text{Ch}(R) \to \text{Ch}(\mathbb{Z}),
 \]
 where $X$ ranges over all members of $\mathcal{X}$.  The (ordinary) Gorenstein Projective modules $\mathcal{GP}$ is the special case where $\mathcal{X}$ is the class of projective modules.

 \item\label{item_Def_XGF} For a fixed class $\mathcal{X}$ of (right) modules:  $\mathcal{X}$-$\mathcal{GF}$ is $\text{Cyc}\left(\text{Ac}(\mathcal{C}_{\text{flat}}, \mathcal{X} \otimes -  )\right)$, where $\mathcal{C}_{\text{flat}}$ is the class of acyclic complexes with flat entries, and $\mathcal{X} \otimes -$ is the collection of (covariant) functors of the form 
 \[
 X \otimes -: \text{Ch}(R) \to \text{Ch}(\mathbb{Z}),
 \]
 where $X$ ranges over all members of $\mathcal{X}$.  The (ordinary) Gorenstein Flat modules $\mathcal{GF}$ is the special case where $\mathcal{X}$ is the class of (right) injective modules.
 
\end{enumerate}

To re-prove Theorem \ref{thm_EstradaEtAl} we will use:
\begin{theorem}[Yang-Liang~\cite{MR3178060}, pp. 3083--3084]\label{thm_YangLiangDeconPrecover}
Suppose $\mathcal{D}$ is a precovering class in $\text{Ch}(R)$, and each member of $\mathcal{D}$ is acyclic.  Then $\text{Cyc}(\mathcal{D})$ is a precovering class in $\text{Ch}(R)$.
\end{theorem}
\noindent The proof of Theorem \ref{thm_YangLiangDeconPrecover} is a purely algebraic proof, showing how given any module $M$, a $\mathcal{D}$-precover of the complex
\begin{tikzcd}
\dots \arrow[r]  & 0 \arrow[r] & M \arrow[r] & 0 \arrow[r] & \dots
\end{tikzcd}
in $\text{Ch}(R)$ can be used to define a $\text{Cyc}(\mathcal{D})$-precover of $M$ in $R$-Mod.  See \cite{MR3178060} for details.

Now we prove the Estrada-Iacob-P\'{e}rez theorem:
\begin{proof}
(of Theorem \ref{thm_EstradaEtAl}) By definition, $\mathcal{X} \text{-} \mathcal{GF} = \text{Cyc}\Big( \text{Ac}(\mathcal{C},\mathcal{F}) \Big)$, where:
\begin{itemize}
 \item $\mathcal{C}$ are the acyclic complexes with flat entries;
 \item $\mathcal{F}$ is the collection of functors of the form $X \otimes -$, for $X \in \mathcal{X}$.
\end{itemize}
We will prove that $\text{Ac}(\mathcal{C},\mathcal{F})$ is deconstructible; then by Theorem \ref{thm_DeconPrecover} it will be precovering (in $\text{Ch}(R)$), and finally by Theorem \ref{thm_YangLiangDeconPrecover}, $\text{Cyc}\big( \text{Ac}(\mathcal{C},\mathcal{F}) \big)$ (= $\mathcal{X}$-$\mathcal{GF}$) will be precovering in $R$-Mod.  

Suppose $F_\bullet \in \mathcal{C}$; i.e., $F_\bullet$ is an acyclic complex of flat modules.  For each $n \in \mathbb{Z}$ let $C_n:= \text{ker}(F_n \to F_{n+1})$ (the $n$-th cycle).  Then the following equivalences hold, via purely homological arguments, where $X^+$ denotes the character dual $\text{Hom}_{\mathbb{Z}}(X,\mathbb{Q}/\mathbb{Z})$:\footnote{The author thanks Jan \v{S}aroch for explaining these equivalences to him.}
\begin{align*}
& F_\bullet \in \text{Ac}(\mathcal{C},\mathcal{F}) &  \\
\iff & \forall X \in \mathcal{X} \ X \otimes F_\bullet  \text{ is acyclic} & \text{(by definition)} \\
\iff & \forall n \in \mathbb{Z} \ \forall X \in \mathcal{X} \ \text{Tor}\left( X, C_n \right)=0 & \text{(using flatness of the } F_n \text{'s)} \\
\iff & \forall n \in \mathbb{Z} \  \forall X \in \mathcal{X} \  \text{Ext}(C_n,X^+)=0  & \text{(by Ext-Tor relations, \cite{MR2985554} Lemma 2.16)} \\
\iff & \forall n \in \mathbb{Z} \ C_n \in {}^\perp \underbrace{\{ X^+ \ : \ X \in \mathcal{X}  \}}_{=: \mathcal{Y}}  & \text{(by definition)}. 
\end{align*}
By Eklof's Lemma \ref{lem_EklofLemma}, the class ${}^\perp \mathcal{Y}$ is filtration closed in $R$-Mod, and this easily implies (via the equivalences above) that $\text{Ac}(\mathcal{C},\mathcal{F})$ is filtration-closed in $\text{Ch}(R)$.  So it remains to show that $\text{Ac}(\mathcal{C},\mathcal{F})$ is \emph{weakly} $|R|^+$-deconstructible in $\text{Ch}(R)$.  Since character duals are pure-injective (Enochs-Jenda~\cite{MR1753146}, Proposition 5.3.7), the Eklof-Trlifaj Theorem \ref{thm_EklofTrlifaj} implies that 
\begin{equation}
{}^\perp \mathcal{Y} \text{ is } |R|^+ \text{-deconstructible in } R\text{-Mod}.
\end{equation}
Consider any $F_\bullet \in \text{Ac}(\mathcal{C},\mathcal{F})$ and assume
\[
F_\bullet \in \mathfrak{N} \prec (H_\theta,\in, \mathcal{Y} \cap H_\theta),
\]
where $\mathfrak{N} \cap |R|^+$ is transitive.  Now each $C_n=\text{ker} (F_n \to F_{n+1})$ is an element of $\mathfrak{N} \cap {}^\perp \mathcal{Y}$, so by deconstructibility of ${}^\perp \mathcal{Y}$ and Theorem \ref{thm_DeconChar}, 
\[
\forall n \in \mathbb{Z} \ \mathfrak{N} \cap C_n \in {}^\perp \mathcal{Y} \text{ and } \frac{C_n}{\mathfrak{N} \cap C_n} \in {}^\perp \mathcal{Y}.
\]
Moreover, the $n$-th cycles of $F_\bullet \restriction \mathfrak{N}$ and $F_\bullet / \mathfrak{N}$ are, respectively, $\mathfrak{N} \cap C_n$ and $\frac{C_n}{\mathfrak{N} \cap C_n}$.  So all cycles of $F_\bullet \restriction \mathfrak{N}$ and $F_\bullet / \mathfrak{N}$ lie in ${}^\perp \mathcal{Y}$, so by the equivalences above, both complexes are in $\text{Ac}(\mathcal{C},\mathcal{F})$.
\end{proof}

The next theorem, in conjunction with Theorem \ref{thm_TurnStatToClub}, will be convenient when trying to show that various classes in Gorenstein Homological Algebra are deconstructible.

\begin{theorem}\label{thm_GorComplexAEclosure}

Suppose $\mathcal{C}$ is a class of acyclic complexes in $\text{Ch}(R)$, $\mathcal{F}$ is a collection of functors from $\text{Ch}(R) \to \text{Ch}(\mathbb{Z})$, and:
\begin{enumerate}
 \item\label{item_BackgroundComplexAE} $\mathcal{C}$ is eventually a.e.\ closed under quotients in $\text{Ch}(R)$ (recall this holds if, for example, $\mathcal{C}$ is deconstructible);
 \item\label{item_F_actsExactlyC} For every short exact sequence of the form
 \[
 \begin{tikzcd}
\mathcal{E}: &  0_\bullet \arrow[r] & C_\bullet \restriction \mathfrak{N} \arrow[r, "\text{incl}"] & C_\bullet \restriction \mathfrak{N}' \arrow[r, "p"] & \frac{C_\bullet \restriction \mathfrak{N}'}{C_\bullet \restriction \mathfrak{N}} \arrow[r] & 0_\bullet
 \end{tikzcd}
 \]
 with $\{ C_\bullet, R \} \cup R \subset \mathfrak{N} \prec \mathfrak{N}' \prec (H_\theta,\in)$ and both $\mathfrak{N} \cap \kappa$ and $\mathfrak{N}' \cap \kappa$ transitive: \textbf{if}  all three inner entries of $\mathcal{E}$ are in $\mathcal{C}$, \textbf{then} each $F \in \mathcal{F}$ preserves the exactness of $\mathcal{E}$; i.e.
 \[
 \begin{tikzcd}
F(\mathcal{E}): &  F(0_\bullet) \arrow[r, no head] & F(C_\bullet \restriction \mathfrak{N}) \arrow[r, , no head, "F(\text{incl})"] & F(C_\bullet \restriction \mathfrak{N}') \arrow[r, no head,  "F(p)"] & F\left( \frac{C_\bullet \restriction \mathfrak{N}'}{C_\bullet \restriction \mathfrak{N}} \right) \arrow[r, no head] & F(0_\bullet)
 \end{tikzcd}
 \]
---with the direction of the arrows depending on whether $F$ is covariant or contravariant---is a short exact sequence in $\text{Ch}(\mathbb{Z})$. 
\end{enumerate}
Then $\text{Ac}(\mathcal{C},\mathcal{F})$ is eventually a.e.\ closed under quotients in $\text{Ch}(R)$.
\end{theorem}

\begin{proof}
This is essentially due to the 3-by-3 lemma.  Fix any regular $\kappa > |R|$ witnessing eventual a.e.\ closure of $\mathcal{C}$ under quotients.  Suppose 
\[
C_\bullet \in \mathfrak{N} \prec \mathfrak{N}' \prec (H_\theta,\in,\kappa,R,  \mathcal{C} \cap H_\theta ),
\]
$\mathfrak{N}$ and $\mathfrak{N}'$ both have transitive intersection with $\kappa$, and that 
\begin{equation}\label{eq_AssumptionsForQuotient}
\Big\{  C_\bullet, \ C_\bullet \restriction \mathfrak{N}, \ C_\bullet \restriction \mathfrak{N}'  \Big\} \subset \text{Ac}(\mathcal{C},\mathcal{F}).
\end{equation} 
We need to prove that $\frac{C_\bullet \restriction \mathfrak{N}'}{C_\bullet \restriction \mathfrak{N} } \in \text{Ac}(\mathcal{C},\mathcal{F})$.  Since $\mathcal{C}$ is $\kappa$-a.e.\ closed under quotients, and $\text{Ac}(\mathcal{C},\mathcal{F}) \subseteq \mathcal{C}$ by definition, \eqref{eq_AssumptionsForQuotient} implies that the complexes $C_\bullet \restriction \mathfrak{N}$, $C_\bullet \restriction \mathfrak{N}'$, and $\frac{C_\bullet \restriction \mathfrak{N}'}{C_\bullet \restriction \mathfrak{N} }$ are all in $\mathcal{C}$.  So it only remains to prove that $\frac{C_\bullet \restriction \mathfrak{N}'}{C_\bullet \restriction \mathfrak{N} }$ remains acyclic after applying any functor $F \in \mathcal{F}$.  Consider the short exact sequence
\begin{equation*}
\begin{tikzcd}
\mathcal{E}: & 0_\bullet \arrow[r] & C_\bullet \restriction \mathfrak{N} \arrow[r, "i"] & C_\bullet \restriction \mathfrak{N}' \arrow[r, "p"] & \frac{C_\bullet \restriction \mathfrak{N}'}{C_\bullet \restriction \mathfrak{N}} \arrow[r] & 0_\bullet
\end{tikzcd}
\end{equation*}
with $i$ the inclusion.  Since each of the three inner terms are, by assumption, in $\text{Ac}(\mathcal{C},\mathcal{F})$ and in particular in $\mathcal{C}$, assumption \eqref{item_F_actsExactlyC} ensures that \begin{equation*}
\begin{tikzcd}
F(\mathcal{E}): & 0_\bullet \arrow[r, no head] & F(C_\bullet \restriction \mathfrak{N}) \arrow[r, no head, "F(i)"] & F(C_\bullet \restriction \mathfrak{N}') \arrow[r, no head, "F(p)"] & F \left( \frac{C_\bullet \restriction \mathfrak{N}'}{C_\bullet \restriction \mathfrak{N}}\right) \arrow[r, no head] & 0_\bullet
\end{tikzcd}
\end{equation*}
is a short exact sequence in $\text{Ch}(\mathbb{Z})$ (with the direction of the arrows depending on whether $F$ is contravariant or covariant).

By assumption \eqref{eq_AssumptionsForQuotient}, both $C_\bullet \restriction \mathfrak{N}$ and $C_\bullet \restriction \mathfrak{N}'$ are in $\text{Ac}(\mathcal{C},\mathcal{F})$, so both $F(C_\bullet \restriction \mathfrak{N})$ and $F(C_\bullet \restriction \mathfrak{N}')$ are acyclic.  Then by the 3-by-3 lemma and exactness of $F(\mathcal{E})$, it follows that $F \left( \frac{C_\bullet \restriction \mathfrak{N}'}{C_\bullet \restriction \mathfrak{N}}\right)$ is acyclic.
\end{proof}

\begin{lemma}\label{lem_GorComplexesSatisfy}
The classes of complexes $\text{Ac}(\mathcal{C}_{\text{proj}}, \text{Hom}(-,\mathcal{X})  )$ and $\text{Ac}(\mathcal{C}_{\text{flat}}, \mathcal{X} \otimes -  )$---i.e., those used to define $\mathcal{X}$-$\mathcal{GP}$ and $\mathcal{X}$-$\mathcal{GF}$ (resp.) in \ref{item_Def_XGP} and \ref{item_Def_XGF} on page \pageref{item_Def_XGP}---satisfy the assumptions of Theorem \ref{thm_GorComplexAEclosure}.
\end{lemma}

\begin{proof}
Since the flat modules and projective modules are deconstructible in $R$-Mod, a routine use of Theorems \ref{thm_DeconCharComplexes} and \ref{thm_DeconChar} implies that the classes $\mathcal{C}_{\text{proj}}$ and $\mathcal{C}_{\text{flat}}$ are deconstructible in $\text{Ch}(R)$.  So by Lemma \ref{lem_DeconImpliesClosure}, they are eventually a.e.\ closed under quotients.  So both classes satisfy clause \ref{item_BackgroundComplexAE} of the assumptions of Theorem \ref{thm_GorComplexAEclosure}.

To see that $\text{Ac}(\mathcal{C}_{\text{proj}}, \text{Hom}(-,\mathcal{X})  )$ satisfies clause \ref{item_F_actsExactlyC} of Theorem \ref{thm_GorComplexAEclosure}, suppose $\kappa > |R|$, $P_\bullet$ is an acyclic complex of projective modules, $\{ C_\bullet, R \} \subset \mathfrak{N} \prec \mathfrak{N}' \prec (H_\theta,\in)$, and $\mathfrak{N}$ and $\mathfrak{N}'$ have transitive intersection with $\kappa$.  Consider the short exact sequence
\begin{equation*}
\begin{tikzcd}
\mathcal{E}: & 0_\bullet \arrow[r] & C_\bullet \restriction \mathfrak{N} \arrow[r, "i"] & C_\bullet \restriction \mathfrak{N}' \arrow[r, "p"] & \frac{C_\bullet \restriction \mathfrak{N}'}{C_\bullet \restriction \mathfrak{N}} \arrow[r] & 0_\bullet
\end{tikzcd}
\end{equation*}
with $i$ the inclusion.  By Theorem \ref{thm_Kaplansky}, for each $n \in \mathbb{Z}$, the module $\frac{\mathfrak{N}' \cap C_n}{\mathfrak{N} \cap C_n}$ is projective; so $\mathcal{E}$ is ``degreewise" split.  Then for any $X$ and any $n \in \mathbb{Z}$, $\text{Hom}(-,X)$ preserves exactness of $\mathcal{E}$ at degree $n$.  Hence, $\text{Hom}(\mathcal{E},X)$ is a short exact sequence in $\text{Ch}(\mathbb{Z})$.

That $\text{Ac}(\mathcal{C}_{\text{flat}}, \mathcal{X} \otimes -  )$ satisfies clause \ref{item_F_actsExactlyC} of Theorem \ref{thm_GorComplexAEclosure} follows from a similar argument, except one uses instead that $\mathcal{E}$ is pure in each degree (by Lemma \ref{lem_ElemPureReflect}), so that tensoring with any module preserves exactness.\footnote{Flatness of entries in $\mathcal{C}_{\text{flat}}$ is actually irrelevant here; if $\mathcal{C}$ is \emph{any} deconstructible class in $\text{Ch}(R)$, then $\text{Ac}(\mathcal{C}, \mathcal{X} \otimes -  )$ will be deconstructible in $\text{Ch}(R)$.}

\end{proof}

To prove Theorem \ref{thm_GP_ZFC}, we will need a variant of the Yang-Liang Theorem \ref{thm_YangLiangDeconPrecover}:
\begin{theorem}\label{thm_WeakDeconCycles}
Suppose $\mathcal{D}$ is a deconstructible class in $\text{Ch}(R)$.  Then $\text{Cyc}(\mathcal{D})$ is \textbf{weakly} deconstructible in $R$-Mod (but possibly not filtration-closed).
\end{theorem}
\begin{proof}
Say $\kappa > |R|$ witnesses deconstructibility of $\mathcal{D}$.  Then by Theorem \ref{thm_DeconCharComplexes}, 
\[
D_\bullet \in \mathfrak{N} \cap \mathcal{D} \ \implies \left( D_\bullet \restriction \mathfrak{N} \in \mathcal{D} \text{ and } D_\bullet / \mathfrak{N} \in \mathcal{D} \right)
\]
whenever $\mathfrak{N} \prec (H_\theta,\in,\kappa, \mathcal{D} \cap H_\theta)$ and $\mathfrak{N} \cap \kappa$ is transitive.

To prove weak deconstructibility of $\text{Cyc}(\mathcal{D})$, we use the \ref{item_non_diag} $\implies$ \ref{item_KappaDecon} direction of Theorem \ref{thm_DeconChar} (which, by Remark \ref{rem_DontNeedFiltClosure}, does \emph{not} need filtration-closure).  Consider any $C \in \text{Cyc}(\mathcal{D})$, and fix a $D_\bullet \in \mathcal{D}$ such that $C = \text{Ker} (D_0 \to D_1)$.  Fix a $\theta$ such that $C,D_\bullet \in H_\theta$.  Set
\[
\mathfrak{A}:=(H_\theta,\in,\kappa,C,D_\bullet, \mathcal{D} \cap H_\theta)
\]
and consider any $\mathfrak{N}  \prec \mathfrak{A}$ whose intersection with $\kappa$ is transitive.  Then $D_\bullet \restriction \mathfrak{N}$ and $D_\bullet / \mathfrak{N}$ are both in $\mathcal{D}$, and hence their respective 0-th cycles are in $\text{Cyc}(\mathcal{D})$.  The 0th cycle of $D_\bullet \restriction \mathfrak{N}$ is $\mathfrak{N} \cap C$, and the 0-th cycle of $D_\bullet / \mathfrak{N}$ is $\frac{C}{\mathfrak{N} \cap C}$.
\end{proof}

Finally, we prove:
\begin{proof}
(of Theorem \ref{thm_GP_ZFC})  Assume $\lambda$ is an infinite regular cardinal and $\mathcal{X}$ is a class of $\lambda$-pure-injective modules.  By definition, 
\[
\mathcal{X} \text{-} \mathcal{GP} = \text{Cyc}\Big( \text{Ac}(\mathcal{C},\mathcal{F}) \Big)
\]
where:
\begin{itemize}
 \item $\mathcal{C}$ are the acyclic complexes with projective entries;
 \item $\mathcal{F}$ is the collection of functors of the form $\text{Hom}(-, X)$, for $X \in \mathcal{X}$.
\end{itemize}
The class $\mathcal{X}$-$\mathcal{GP}$ is known to be filtration-closed; the homological argument from Enochs-Iacob-Overtoun~\cite{MR2342674} Theorem 3.2, for the case $\mathcal{X} = \text{Proj}$, works for any $\mathcal{X}$.  So it remains to show that $\mathcal{X}$-$\mathcal{GP}$ ($= \text{Cyc}\big( \text{Ac}(\mathcal{C},\mathcal{F} \big)$) is weakly deconstructible, and by Theorem \ref{thm_WeakDeconCycles} it suffices to prove that $\text{Ac}(\mathcal{C},\mathcal{F})$ is deconstructible in $\text{Ch}(R)$.  Filtration-closure of $\text{Ac}(\mathcal{C},\mathcal{F})$ follows from the proof of Theorem 2.6 of Enochs-Iacob-Overtoun~\cite{MR2342674} argument, so it only remains to show weak deconstructibility of $\text{Ac}(\mathcal{C},\mathcal{F})$ in $\text{Ch}(R)$. 

By Theorem \ref{thm_GorComplexAEclosure} and Lemma \ref{lem_GorComplexesSatisfy}, $\text{Ac}(\mathcal{C},\mathcal{F})$ is eventually a.e.\ closed under quotients.  Fix a regular $\kappa$ such that $\text{Ac}(\mathcal{C},\mathcal{F})$ is $\kappa$-a.e.\ closed under quotients and $\alpha^{<\lambda} < \kappa$ for all $\alpha < \kappa$.  So by Theorem \ref{thm_TurnStatToClub}, it suffices to find a \textbf{stationary} set of $\mathfrak{N} \in \wp_\kappa(H_\theta)$ such that $\mathfrak{N}$ reflects membership in $\text{Ac}(\mathcal{C},\mathcal{F})$; i.e., such that 
\[
P_\bullet \in \mathfrak{N} \cap \text{Ac}\big(\mathcal{C},\mathcal{F} \big) \ \implies \ P_\bullet \restriction \mathfrak{N} \in \text{Ac}\big(\mathcal{C},\mathcal{F} \big).
\]
We will show this holds for any $<\lambda$-closed $\mathfrak{N}$, the collection of which constitutes a stationary set by Fact \ref{fact_BasicFactsElemSub}.

For any acyclic complex $P_\bullet$ of projective modules with cycles $C_n:= \text{ker}(P_n \to P_{n+1})$, we have the following equivalences:
\begin{align*}
& P_\bullet \in \text{Ac}(\mathcal{C},\mathcal{F}) & \\
\iff & \forall X \in \mathcal{X} \ \text{Hom}(P_\bullet,X) \text{ is acyclic} & \text{(by definition)} \\ 
\iff &  \forall X \in \mathcal{X} \ \forall n \in \mathbb{Z} \ C_n \in {}^\perp X & \text{(using projectivity of the } P_n \text{'s)} \\
\iff &  \forall n \in \mathbb{Z} \ C_n \in {}^\perp \mathcal{X} & \text{(by definition)}.  
\end{align*}

Consider any $<\lambda$-closed $\mathfrak{N} \prec (H_\theta,\in,\mathcal{X} \cap H_\theta)$ whose intersection with $\kappa$ is transitive, and suppose $P_\bullet \in \mathfrak{N}$ is a member of $\text{Ac}(\mathcal{C},\mathcal{F})$.  Then $C_n \in \mathfrak{N} \cap {}^\perp \mathcal{X}$ for all $n \in \mathbb{Z}$.  Since $\mathcal{X}$ consists of $\lambda$-pure-injective modules, and by $<\lambda$-closure of $\mathfrak{N}$,  Lemma \ref{lem_LambdaClosedLeftPerp} ensures that 
\[
\forall n \in \mathbb{Z} \ \underbrace{\mathfrak{N} \cap C_n}_{\text{ the } n \text{-th cycle of } P_\bullet \restriction \mathfrak{N}} \in {}^\perp \mathcal{X}  
\]
(Lemma \ref{lem_LambdaClosedLeftPerp} also tells us $\frac{C_n}{\mathfrak{N} \cap C_n} \in {}^\perp \mathcal{X}$, but we don't need that here, because we only need to check the assumptions of Theorem \ref{thm_TurnStatToClub}).  Also, $P_\bullet \restriction \mathfrak{N}$ consists of projective modules, by Kaplansky's Theorem \ref{thm_Kaplansky}.  So by the equivalences above, $P_\bullet \restriction \mathfrak{N} \in \text{Ac}(\mathcal{C},\mathcal{F})$.
\end{proof}

\begin{remark}
The proof does not actually need projectivity of the complexes.  The argument works as long as $\mathcal{C}$ is a deconstructible class of complexes with entries from ${}^\perp \mathcal{X}$.
\end{remark}

\subsubsection{Maximum Deconstructibility and Gorenstein Homological Algebra}\label{sec_MaxDecon}

\begin{theorem}[Cort{\'e}s-Izurdiaga and \v{S}aroch \cite{cortes2023module}]\label{thm_CS_LC_GP}
If there is a proper class of strongly compact cardinals, then over every ring, the class of Gorenstein Projective modules is deconstructible.
\end{theorem}

The author independently proved the same conclusion, but under the stronger assumption of a proper class of supercompact cardinals (Cox~\cite{Cox_MaxDecon}).  The proofs from \cite{cortes2023module} and \cite{Cox_MaxDecon} are very different, and generalize in different ways; see the introduction of \cite{Cox_MaxDecon} for a discussion.  The argument from \cite{Cox_MaxDecon} uses (of course) elementary submodel methods, and generalizes to show:
\begin{theorem}[\cite{Cox_MaxDecon}]\label{thm_VP_XGP}
Vop\v{e}nka's Principle implies that all classes of the form $\mathcal{X}$-$\mathcal{GP}$ are deconstructible.
\end{theorem}

Theorem \ref{thm_VP_XGP} follows from an even more powerful theorem.  Recall by Lemma \ref{lem_DeconImpliesClosure} that for any class $\mathcal{C}$ of modules or complexes that is filtration-closed,
\begin{align*}
\mathcal{C} \text{ is deconstructible } \implies \ \mathcal{C} \text{ is eventually a.e. closed under quotients.}
\end{align*}
Under Vop\v{e}nka's Principle, the converse holds:

\begin{theorem}[Cox~\cite{Cox_MaxDecon}]\label{thm_VP_MaxDecon}
Vop\v{e}nka's Principle implies the \textbf{Maximum Deconstructibility} principle, which asserts that for any class $\mathcal{C}$ of modules or complexes:  if $\mathcal{C}$ is filtration-closed and eventually a.e.\ closed under quotients, then it is deconstructible.
\end{theorem}

We refer to \cite{Cox_MaxDecon} for the proof, but roughly: given any (isomorphism-closed and filtration-closed) class $\mathcal{C}$, Vop\v{e}nka's Principle provides many regular $\kappa$ such that for a proper class of $\theta$ there are \textbf{stationarily many} $\mathfrak{N} \in \wp^*_\kappa(H_\theta)$ such that for any $C$,
\[
C \in \mathfrak{N} \cap \mathcal{C} \ \implies \  \mathfrak{N} \cap C \in \mathcal{C}.
\]
But this doesn't quite satisfy clause \ref{item_MainChar} of Theorem \ref{thm_DeconChar},\footnote{Or of Theorem \ref{thm_DeconCharComplexes} in the case where $\mathcal{C}$ is a class of complexes.} because we only have it for \textbf{stationarily many} $\mathfrak{N} \in \wp_\kappa(H_\theta)$, and we are still missing the requirement that the \emph{quotient} $\frac{C}{\mathfrak{N} \cap C}$ is also in $\mathcal{C}$.  And we must assume \textbf{something} about $\mathcal{C}$ if we want to derive its deconstructibility, since there are ZFC examples of filtration-closed $\mathcal{C}$ that are not deconstructible.\footnote{E.g., the $\aleph_1$-free abelian groups, see Section \ref{sec_Decon}.}  But the assumption that $\mathcal{C}$ is eventually a.e.\ closed under quotients provides a way, via Theorem \ref{thm_TurnStatToClub} from page \pageref{thm_TurnStatToClub}, to ``convert" the mere stationary set of $\mathfrak{N} \in \wp^*_\kappa(H_\theta)$ into a club, and moreover obtain the desired quotient behavior.  

To see that Theorem \ref{thm_VP_MaxDecon} implies Theorem \ref{thm_VP_XGP}, fix any ring $R$ and any class $\mathcal{X}$ of $R$-modules.  Using the notation of Section \ref{sec_Gorenstein}, 
\begin{equation}\label{eq_X_GP}
\mathcal{X} \text{-} \mathcal{GP} = \text{Cyc}\left( \text{Ac}(\mathcal{C},\mathcal{F}) \right),
\end{equation}
where $\mathcal{C}$ is the class of acyclic complexes with projective entries, and $\mathcal{F}$ is the collection of functors of the form $\text{Hom}(-,X)$, where $X$ ranges over all members of the class $\mathcal{X}$.  By Theorem \ref{thm_GorComplexAEclosure} and Lemma \ref{lem_GorComplexesSatisfy}, $\text{Ac}(\mathcal{C},\mathcal{F})$ is eventually a.e.\ closed under quotients in $\text{Ch}(R)$.  The class $\text{Ac}(\mathcal{C},\mathcal{F})$ is also (via purely homological arguments) closed under filtrations.  So by the Maximum Deconstructibility Principle, $\text{Ac}(\mathcal{C},\mathcal{F})$ is deconstructible in $\text{Ch}(R)$.  By Theorem \ref{thm_WeakDeconCycles}, the class \eqref{eq_X_GP} is weakly decontructible in $R$-Mod.  But it is also closed under filtrations (by the same argument as Enochs-Iacob-Overtoun~\cite{MR2342674} did when $\mathcal{X} = \text{Proj}$).  So $\mathcal{X}$-$\mathcal{GP}$ is deconstructible.

\section{Salce's Problem}\label{sec_Salce}

Recall the Eklof-Trlifaj Theorem \ref{thm_EklofTrlifaj}, and the variations of it (Theorem \ref{thm_Saroch_Cortes} and \ref{thm_CoxGeneralizeET}), which all dealt with trying to ensure that certain ``root of Ext" classes, i.e., those classes of the form
\[
{}^\perp \mathcal{B} = \{ A \ : \ \forall B \in \mathcal{B} \ \text{Ext}(A,B)=0 \}
\]
for some class $\mathcal{B}$, were deconstructible.  In \cite{Cox_VP_CotPairs}, we proved consistency (relative to Vop\v{e}nka's Principle) that \textbf{every} root of Ext is deconstructible, at least if the ring is, say, a hereditary ring like $\mathbb{Z}$ (a ring is hereditary if submodules of projective modules are always projective):  
\begin{theorem}[Cox~\cite{Cox_VP_CotPairs}]\label{thm_Cox_Salce}
If Vop\v{e}nka's Principle is consistent with ZFC, then it is consistent with ZFC that for every ring $R$:\footnote{The listed conclusions follow from the conjunction of Vop\v{e}nka's Principle with Jensen's Diamond principle holding at every stationary subset of every regular uncountable cardinal.  This can be forced over any model of Vop\v{e}nka's Principle, using a preservation theorem of Brooke-Taylor~\cite{MR2805294}.}  if
\begin{quote}
${}^\perp M$ is downward closed under elementary submodules for all $R$-modules $M$,\footnote{This is true for any \emph{hereditary} ring, such as $\mathbb{Z}$.  \v{S}aroch has since weakened this assumption on $R$.} 
 \end{quote}
 then all cotorsion pairs in $R$-Mod are both generated and cogenerated by a set.  Hence, for any cotorsion pair $\mathfrak{C}=(\mathcal{A},\mathcal{B})$ in $R$-Mod:
\begin{itemize}
 \item $\mathfrak{C}$ is complete\footnote{This means that every $R$-module is of the form $A/B$ for some $A \in \mathcal{A}$ and $B \in \mathcal{B}$.} (by \cite{MR1798574}), and
 \item $\mathcal{A}$ is deconstructible (by \cite{MR2985554}, Theorem 7.13)
\end{itemize}
\end{theorem}

Coupled with an earlier result of Eklof-Shelah~\cite{MR2031314}, Theorem \ref{thm_Cox_Salce} shows that a problem of Salce~\cite{MR565595}---whether every cotorsion pair in \textbf{Ab} is complete---is independent of ZFC.  And, while it is apparently still open (see e.g., Sl\'{a}vik-Trlifaj~\cite{MR3161764}) whether ZFC proves the existence of a non-deconstructible \emph{root of Ext} (i.e., class of the form ${}^\perp \mathcal{B}$), Theorem \ref{thm_Cox_Salce} gives a negative answer for hereditary rings.  Again, combined with Eklof-Shelah~\cite{MR2031314}, this shows that the statement
\begin{quote}
``There exists, over some hereditary ring, a non-deconstructible root of Ext"
\end{quote} 
is independent of ZFC.

Large cardinals have also appeared recently in settings that are tangentially related to deconstructibility, such as Ben Yassine-Trlifaj~\cite{MR4817675}, Cox-Rosick\'{y}~\cite{CoxRosicky}, and \v{S}aroch-Trlifaj~\cite{MR4186456}.

\section{Questions}\label{sec_Questions}

Recall the Eklof-Trlifaj Theorem \ref{thm_EklofTrlifaj}, that ${}^\perp \mathcal{B}$ is deconstructible whenever $\mathcal{B}$ consists of pure-injective modules.  By Remark \ref{rem_WeakFormPureInj}, our new proof of their theorem works under the apparently weaker assumption that members of $\mathcal{B}$ are ``merely" injective with respect to very special kinds of pure embeddings; namely, when members of $\mathcal{B}$ are injective with respect to those inclusions of the form $\mathfrak{N} \cap A \to A$ where $\{ A,R \} \cup R \subset \mathfrak{N} \prec (H_\theta,\in)$.

It is unclear if this really results in a stronger theorem though.  To clarify the problem, recall that for a class $\mathcal{M}$ of monomorphisms, a module $B$ is defined to be $\boldsymbol{\mathcal{M}}$\textbf{-injective} if for every map $m: X \to Y$ in $\mathcal{M}$, the canonical map $m^*: \text{Hom}\big(Y,B \big) \to \text{Hom} \big(X,B \big)$ is surjective; i.e., every homomorphism $X \to B$ factors through $m$:
\[
\begin{tikzcd}
X \arrow[r, hook, "m"] \arrow[d] & Y \arrow[dl, dotted] \\
B & 
\end{tikzcd}
\]
Given a class $\mathcal{M}$ of monomorphisms, let $\mathcal{M}_\ell$ denote the class of morphisms $i$ such that $m = j \circ i$ for some $m \in \mathcal{M}$ and some monomorphism $j$.\footnote{E.g., if $m: X \to Y$ is an inclusion with $m \in \mathcal{M}$, and $X \subset Z \subset Y$, then the inclusion $X \to Z$ is a member of $\mathcal{M}_\ell$.}  And let $\widetilde{M}$ denote the closure under isomorphisms of $\mathcal{M}_\ell$ in the arrow category.  It is easy to check that $B$ is $\mathcal{M}$-injective if and only if $B$ is $\widetilde{\mathcal{M}}$-injective.  And if $\mathcal{M}$ is the class of pure embeddings, then $\mathcal{M} = \widetilde{\mathcal{M}}$.

For a ring $R$, define
\[
\mathcal{STT}_R:=\Big\{ \mathfrak{N} \cap A \to_{\text{incl}} A \ : \ \left( \exists \theta \in \text{Reg} \right) \left(  \{ A, R \} \cup R \subset \mathfrak{N} \prec (H_\theta,\in) \right) \Big\}.
\]
($\mathcal{STT}$ stands for ``set-theoretic trace").  In general,

\begin{equation}\label{eq_STT_strictlyContained}
\widetilde{\mathcal{STT}}_R \subsetneq \{ \text{pure embeddings of R-modules}  \}.
\end{equation}

\noindent The inclusion follows from Lemma \ref{lem_ElemPureReflect}.  But the inclusion is strict.  For example, consider the pure embedding $i: \mathbb{Z} \to \mathbb{Z} \oplus \mathbb{Z}$, $n \mapsto n \oplus 0$ in \textbf{Ab}.  We claim $i \notin \widetilde{\mathcal{STT}}$.  Suppose otherwise; then there would exist a commutative diagram
\[
\begin{tikzcd}
\mathbb{Z} \arrow[r, "i"] \arrow[d, "\simeq"] &  \mathbb{Z} \oplus \mathbb{Z} \arrow[d, "\simeq"] & \\
\mathfrak{N} \cap A \arrow[r, "\text{incl}"]  & B \arrow[r, "\text{incl}"] & A
\end{tikzcd}
\]
of homomorphisms whose vertical arrows are isomorphisms, and such that $A \in \mathfrak{N} \prec (H_\theta,\in)$.  But $\mathfrak{N} \cap A$ being isomorphic to $\mathbb{Z}$, together with the assumption that $A \in \mathfrak{N}$, implies $A = \mathfrak{N} \cap A \simeq \mathbb{Z}$.\footnote{Since $\mathfrak{N} \cap A \simeq \mathbb{Z}$, $\mathfrak{N} \cap A = \langle a \rangle$ for some $a \in \mathfrak{N} \cap A$.  Then $\mathfrak{N} \models$ ``for every $b \in A$, $b=na$ for some $n \in \mathbb{Z}$".  Since $\mathfrak{N} \prec (H_\theta,\in)$, this statement is true in $(H_\theta,\in)$ too.  So $A = \langle a \rangle = \mathfrak{N} \cap A$.}  But this is a contradiction, since $A\simeq \mathbb{Z}$ cannot contain a subgroup isomorphic to $\mathbb{Z} \oplus \mathbb{Z}$.

Even though \eqref{eq_STT_strictlyContained} is strict, it is unclear whether the following containment is strict: 
\begin{equation}\label{eq_PI_ETI}
\{ \text{ pure injective modules } \} \subseteq \{ \mathcal{STT}\text{-injective modules} \}.
\end{equation}
\noindent (Note the right side is the same as the class of $\widetilde{\mathcal{STT}}$-injective modules).

\begin{question}\label{q_STT_inj}
Is the inclusion \eqref{eq_PI_ETI} strict over some rings? 
\end{question}

An affirmative answer would mean that Remark \ref{rem_WeakFormPureInj} really does yield a strengthening of the Eklof-Trlifaj Theorem \ref{thm_EklofTrlifaj}.

\begin{question}\label{q_GP_LC}
Is the class $\mathcal{GP}_R$ always precovering (or always deconstructible) over every ring?  Assuming a proper class of strongly compact cardinals, the answer is yes (Cort{\'e}s-Izurdiaga and \v{S}aroch \cite{cortes2023module}).  It is not known whether ``there exists a ring $R$ such that $\mathcal{GP}_R$ is either not precovering, or not deconstructible" is even consistent with ZFC.
\end{question}

By the argument described in Section \ref{sec_MaxDecon}, to prove that $\mathcal{GP}_R$ is deconstructible, it would suffice to find a $\kappa > |R|$ and stationarily many $\mathfrak{N} \in \wp^*_\kappa(H_\theta)$ such that for any acyclic complex $P_\bullet \in \mathfrak{N}$ of projective modules, 
\[
P_\bullet \text{ is } \text{Hom}\left(-,\text{Proj} \right) \text{-acyclic} \ \implies \ P_\bullet \restriction \mathfrak{N} \text{ is } \text{Hom}\left(-,\text{Proj} \right) \text{-acyclic}.  
\]

\begin{question}
What is the exact consistency strength of the scheme ``Every cotorsion pair of classes of abelian groups is complete", or of the scheme ``Every class of roots of Ext in \textbf{Ab} is deconstructible"?  
\end{question}

By Cox~\cite{Cox_VP_CotPairs}, an upper bound is the consistency of Vop\v{e}nka's Principle.  By Eklof-Shelah~\cite{MR2031314}, a lower bound is failure of a kind of uniformization principle.

\appendix

\section{Simplified proof of the characterization of deconstructibility}\label{app_ProofDeconChar}

We prove the equivalence of \ref{item_KappaDecon} with \ref{item_MainChar} from Theorem \ref{thm_DeconChar}, which avoids some of the technical complications of the original proof from \cite{Cox_MaxDecon}.  These technical complications are avoided because of our assumption here that $|R|<\kappa$.  See \cite{Cox_MaxDecon} for more details, including how to get further equivalence with clauses \ref{item_DiagVersion2models} and \ref{item_non_diag}, and how to deal with situations where $\kappa \le |R|$ by utilizing the characterization of $\kappa$-Noetherian rings from Theorem \ref{thm_CharNoeth}.

\subsection{Proof of \ref{item_MainChar} $\implies$ \ref{item_KappaDecon}}

Assume \ref{item_MainChar}; we prove that $\mathcal{C}$ is weakly $\kappa$-deconstructible; by the background assumption that $\mathcal{C}$ is closed under filtrations, this will imply (by Lemma \ref{lem_FiltOfWeakDecon}) that $\mathcal{C}$ is $\kappa$-deconstructible.  

We prove by induction on $|C|$ that every $C \in \mathcal{C}$ has a $\mathcal{C}_{\kappa}$-filtration.  For $|C|<\kappa$ this is trivial.  Now assume $\lambda \ge \kappa$ and that every member of $\mathcal{C}$ of size $<\lambda$ has a $\mathcal{C}_\kappa$-filtration.  Fix a $C \in \mathcal{C}$ of size $\lambda$, fix a regular $\theta > \lambda$, and (using the assumption that $\lambda \ge \kappa$ and $\kappa$ is regular) use the Downward L\"owenheim-Skolem Theorem to recursively build a $\subseteq$-increasing, and $\subseteq$-continuous sequence
\[
\langle \mathfrak{N}_i \ : \ i < \text{cf}(\lambda) \rangle
\]
of elementary submodels of $(H_\theta,\in, \kappa,R, \mathcal{C} \cap H_\theta)$ such that $C \in \mathfrak{N}_0$, each $\mathfrak{N}_i \cap \kappa$ is transitive, each $\mathfrak{N}_i$ is of size $<\lambda$, each $\mathfrak{N}_i$ is an element (and subset) of $\mathfrak{N}_{i+1}$, and $C \subset \bigcup_{i < \text{cf}(\lambda)} \mathfrak{N}_i$.  

Consider any $i < \text{cf}(\lambda)$; since $\mathfrak{N}_i$ and $C$ are both elements of $\mathfrak{N}_{i+1}$, it follows that $\frac{C}{\mathfrak{N}_i \cap C} \in \mathfrak{N}_{i+1}$.  Then 
\begin{equation}\label{eq_i_iplus1_ratio}
\mathfrak{N}_{i+1} \cap \frac{C}{\mathfrak{N}_i \cap C} \ \underset{\text{(Lemma } \ref{lem_IsoQuotient}\text{)}}{\simeq} \ \frac{\mathfrak{N}_{i+1} \cap C}{\mathfrak{N}_{i+1} \cap (\mathfrak{N}_i \cap C)}  \underset{\text{(b/c } \mathfrak{N}_i \subset \mathfrak{N}_{i+1} \text{)}}{=}  \  \frac{\mathfrak{N}_{i+1} \cap C}{\mathfrak{N}_i \cap C}.
\end{equation}
Now $\frac{C}{\mathfrak{N}_i \cap C} \in \mathcal{C}$ by assumption \ref{item_MainChar} of the theorem (since $C \in \mathfrak{N}_i$); and then again by assumption \ref{item_MainChar} of the theorem---since $\frac{C}{\mathfrak{N}_i \cap C}$ is an element of $\mathfrak{N}_{i+1} \cap \mathcal{C}$---the left side of \eqref{eq_i_iplus1_ratio} is an element of $\mathcal{C}$.  So, we have shown that
\begin{equation}
\forall i < \text{cf}(\lambda) \ \ \frac{\mathfrak{N}_{i+1} \cap C}{\mathfrak{N}_i \cap C} \text{ is (isomorphic to) an element of } \mathcal{C}.
\end{equation}
Hence,
\begin{equation}\label{eq_model_filtration}
\left\langle \mathfrak{N}_i \cap C \ : \ i < \text{cf}(\lambda) \right\rangle
\end{equation}
is a $\mathcal{C}$-filtration of $C$.  Furthermore, observe that for each $i$, the quotient $\frac{\mathfrak{N}_{i+1} \cap C}{\mathfrak{N}_i \cap C}$ of adjacent members of the filtration of \eqref{eq_model_filtration} is a $<\lambda$-sized member of $\mathcal{C}$.  So by the induction hypothesis, $\frac{\mathfrak{N}_{i+1} \cap C}{\mathfrak{N}_i \cap C}$ can be $\mathcal{C}_\kappa$-filtered, say by $\vec{Z}^i$.  Then by an argument similar to that in the proof of \eqref{eq_FiltSquared} from page \pageref{eq_FiltSquared}, these $\vec{Z}^i$'s can be used to produce a $\mathcal{C}_\kappa$-filtration of $C$. 

\subsection{Proof of \ref{item_KappaDecon} $\implies$ \ref{item_MainChar}}

Assume \ref{item_KappaDecon}, and consider any 
\[
\mathfrak{N} \prec \mathfrak{A}:=(H_\theta,\in,\kappa,R, \mathcal{C} \cap H_\theta)
\]
such that $\mathfrak{N} \cap \kappa$ is transitive.  Suppose $C \in \mathfrak{N} \cap \mathcal{C}$.  By assumption \ref{item_KappaDecon}, there is a $\mathcal{C}_\kappa$-filtration of $C$.  Since $C \in H_\theta$, such a filtration is necessarily an element of $H_\theta$, and because $\kappa$ and $\mathcal{C} \cap H_\theta$ are included in the structure $\mathfrak{A}$, there is a $\mathcal{C}_\kappa$-filtration $\vec{C} = \langle C_i \ : \ i < \zeta \rangle$ of $C$ such that $\vec{C} \in \mathfrak{N}$.

To see that $\mathfrak{N} \cap C \in \mathcal{C}$, observe that by elementarity of $\mathfrak{N}$, 
\begin{equation}\label{eq_FiltAtN}
\langle \mathfrak{N} \cap C_i \ : \ i \in \mathfrak{N} \cap \zeta \rangle
\end{equation}
is a filtration of $\mathfrak{N} \cap C$:  clearly it is $\subseteq$-increasing, and a routine argument using elementarity of $\mathfrak{N}$ ensures that it is continuous.\footnote{I.e., if $i$ is a limit member of $\mathfrak{N} \cap \zeta$---i.e., $i \in \mathfrak{N} \cap \zeta$ and $i \cap \mathfrak{N}$ has no largest element---then $i$ must be a limit ordinal, and furthermore $\mathfrak{N} \cap C_i = \bigcup_{j \in \mathfrak{N} \cap i} \mathfrak{N} \cap C_j$.}  It remains to check that adjacent quotients of \eqref{eq_FiltAtN} are in $\mathcal{C}$, and then the background assumption that $\mathcal{C}$ is filtration-closed will ensure that its union, which is $\mathfrak{N} \cap C$, is in $\mathcal{C}$.  But by elementarity of $\mathfrak{N}$, if $i$ and $j$ are adjacent ordinals in $\mathfrak{N} \cap \zeta$, $j$ must equal $i+1$.  Furthermore, $|C_{i+1}/C_i|<\kappa$ by assumption, and $\mathfrak{N} \cap \kappa$ is transitive.  Hence $C_j/C_i$ is not only an element of $\mathfrak{N}$, but a subset of it too.  So
\[
\frac{\mathfrak{N} \cap C_{i+1}}{\mathfrak{N} \cap C_i} \underset{\text{(Lemma } \ref{lem_IsoQuotient} \text{)} }{\simeq} \mathfrak{N} \cap \frac{C_{i+1}}{C_i}  \underset{\text{(} C_{i+1}/C_i \subset \mathfrak{N} \text{)}}{=} \frac{C_{i+1}}{C_i} \in \mathcal{C}_\kappa.
\]

To see that $\frac{C}{\mathfrak{N} \cap C}$ is a member of $\mathcal{C}$:  again, by the background assumption that $\mathcal{C}$ is filtration-closed, it suffices to show that $\frac{C}{\mathfrak{N} \cap C}$ is $\mathcal{C}$-filtered.  We will prove that
\[
\left\langle \frac{(\mathfrak{N} \cap C) + C_i}{\mathfrak{N} \cap C} \ : \ i < \zeta \right\rangle
\]
is the desired $\mathcal{C}$-filtration.  Clearly it is a $\subseteq$-increasing and continuous sequence with union $\frac{C}{\mathfrak{N} \cap C}$.  It remains to check that for each $i < \zeta$,
\begin{equation}\label{eq_QuotientToShow}
 \frac{\frac{(\mathfrak{N} \cap C) + C_{i+1}}{\mathfrak{N} \cap C}}{\frac{(\mathfrak{N} \cap C) + C_i}{\mathfrak{N} \cap C}} \text{, which is isomorphic to } \frac{(\mathfrak{N} \cap C) + C_{i+1}}{(\mathfrak{N} \cap C) + C_{i}} \text{, is in } \mathcal{C}. 
\end{equation}

Note that by elementarity of $\mathfrak{N}$, either $i$ and $i+1$ are both elements of $\mathfrak{N}$, or neither $i$ nor $i+1$ is an element of $\mathfrak{N}$.  I.e., the following cases are exhaustive:
\begin{itemize}
 \item \textbf{Case 1: both $i$ and $i+1$ are elements of $\mathfrak{N}$.}  Since $\vec{C} \in \mathfrak{N}$, it follows that both $C_i$ and $C_{i+1}$ are elements of $\mathfrak{N}$.  In this case we show that the quotient in \eqref{eq_QuotientToShow} is trivial.\footnote{This is fine even if $0 \notin \mathcal{C}$, because we can thin out the filtration to eliminate the zero factors. In other words, we can without loss of generality assume $0 \in \mathcal{C}$.}  Since both $C_i$ and $C_{i+1}$ are elements of $\mathfrak{N}$ in this case, and since $|C_{i+1}/C_i|<\kappa$, by elementarity of $\mathfrak{N}$ there is an $X \in \mathfrak{N}$ of size less than $\kappa$ such that $C_{i+1} = X + C_i$.  Since $\mathfrak{N} \cap \kappa$ is transitive, $X$ is also a subset of $\mathfrak{N}$ by Fact \ref{fact_BasicFactsElemSub}.  Hence,
 \[
 \frac{(\mathfrak{N} \cap C) + C_{i+1}}{(\mathfrak{N} \cap C) + C_i} =  \frac{(\mathfrak{N} \cap C) + (X + C_i)}{(\mathfrak{N} \cap C) + C_i} = \frac{(\mathfrak{N} \cap C) +  C_i}{(\mathfrak{N} \cap C) + C_i} =0,
 \]
 where the middle equality is because $X \subset \mathfrak{N} \cap C$.

 \item \textbf{Case 2: neither $i$ nor $i+1$ is an element of $\mathfrak{N}$.}  In this case we show that the quotient in \eqref{eq_QuotientToShow} is isomorphic to $\frac{C_{i+1}}{C_i}$, which is in $\mathcal{C}$ by assumption.  Define
 \[
 \Phi: C_{i+1} \to \frac{(\mathfrak{N} \cap C) + C_{i+1}}{(\mathfrak{N} \cap C) + C_i}
 \]
by mapping $x$ to its coset.  The map $\Phi$ is homomorphic, and is clearly surjective (since $\mathfrak{N} \cap C$ also appears in the denominator of the codomain), so we just need to verify that $\text{Ker} \Phi = C_i$.  That $\text{Ker} \Phi \supseteq C_i$ is obvious.  For the other direction, assume $x \in \text{Ker} \Phi$.  Then
\begin{equation}\label{eq_WhatKerMeans}
x \in C_{i+1} \cap \Big( (\mathfrak{N} \cap C) + C_{i} \Big)
\end{equation}
In particular, $x = n + c_i$ for some $n \in \mathfrak{N} \cap C$ and some $c_i \in C_i$.  Now $x$ and $c_i$ are both elements of $C_{i+1}$, so $n=x-c_i$ is an element of $C_{i+1}$.  So
 \begin{equation}\label{eq_n_N_C_iplus1}
 n \in \mathfrak{N} \cap C_{i+1}.
 \end{equation}
Our case assumption that neither $i$ nor $i+1$ are elements of $\mathfrak{N}$, together with the fact that $\vec{C}$ is a filtration and $\vec{C} \in \mathfrak{N}$, implies that\footnote{Why: suppose $z \in \mathfrak{N} \cap C_{i+1}$, and let $\rho_z$ denote the least ordinal such that $z \in C_{\rho_z}$.  Then $\rho_z$ is definable in $\mathfrak{A}$ from the parameters $z$ and $\vec{C}$, and hence $\rho_z \in \mathfrak{N}$.   Clearly $\rho_z \le i+1$, but by our case assumption, it follows that $\rho_z \le i$ (in fact, $\rho_z < i$). }
\[
\mathfrak{N} \cap C_i = \mathfrak{N} \cap C_{i+1},
\]
and hence
\[
n \in \mathfrak{N} \cap C_i.
\]
And $c_i \in C_i$ by assumption.  Hence, $x = n+c_i$ is a member of $C_i$, completing the proof that $\text{Ker} \Phi \subseteq C_i$.
 
  \end{itemize}

\begin{remark}
We chose to show the proof that \ref{item_KappaDecon} implies \ref{item_MainChar} for simplicity, but the proof that \ref{item_KappaDecon} implies \ref{item_DiagVersion2models} is very similar to the one given above; one just uses $\mathfrak{N}' \cap C$ in place of $C$. 
\end{remark}

\end{document}